\documentclass[reqno]{article}
\usepackage{amsmath,amsfonts,amssymb,amsthm}
\usepackage{bbm}
\usepackage{setspace}

\usepackage{geometry}

\newtheorem{theorem}{Theorem}
\newtheorem{lemma}[theorem]{Lemma}
\newtheorem{claim}[theorem]{Claim}

\theoremstyle{definition}
\newtheorem*{ack}{Acknowledgement}

\newcommand\eps{\varepsilon}
\renewcommand\le{\leqslant}
\renewcommand\ge{\geqslant}

\newcommand\E{{\mathbb E}}
\newcommand\Var{\operatorname{Var}} 
\renewcommand\Pr{{\mathbb P}}
\newcommand{\indic}[1]{\mathbbm{1}_{\{{#1}\}}} 

\newcommand\floor[1]{\lfloor #1 \rfloor}
\newcommand\floorBL[1]{\Bigl\lfloor #1 \Bigr\rfloor}

\newcommand\cD{\mathcal{D}}
\newcommand\cE{\mathcal{E}}
\newcommand\cF{\mathcal{F}}
\newcommand\cH{{\mathcal H}}
\newcommand\cG{{\mathcal G}}
\newcommand\cI{{\mathcal I}}

\newcommand\cP{{\mathcal P}}
\newcommand\cS{{\mathcal S}}
\newcommand\cT{{\mathcal T}}
\newcommand\cU{\mathcal{U}}
\newcommand\cX{\mathcal{X}}
\newcommand\cY{\mathcal{Y}}
\newcommand{\fX}{{\mathfrak X}} 

\newcommand{\vp}{{\mathbf p}}
\newcommand{\ex}{\mathrm{ex}} 
\newcommand{\knk}{K_n^{(k)}} 

\newcommand{\eknk}{E(\knk)} 
\newcommand\bigpar[1]{\bigl(#1\bigr)}

\newcommand\lrpar[1]{\left(#1\right)}

\newcommand\Bin{\operatorname{Bin}}

\begin{document}
%\title{Poisson approximation for large deviations revisited}
\title{The lower tail: Poisson approximation revisited}
\author{Svante Janson%
\thanks{Department of Mathematics, Uppsala University, 
PO Box 480, SE-751~06 Uppsala, Sweden. 
E-mail: {\tt svante.janson@math.uu.se}. Partly supported by the Knut and Alice Wallenberg Foundation.}
\ and Lutz Warnke%
\thanks{Department of Pure Mathematics and Mathematical Statistics,
Wilberforce Road, Cambridge CB3 0WB, UK.
E-mail: {\tt L.Warnke@dpmms.cam.ac.uk}.}}
\date{May 27, 2014}
\maketitle

\begin{abstract}
The well-known ``Janson's inequality'' gives Poisson-like upper bounds 
for the lower tail probability $\Pr(X \le (1-\eps)\E X)$ when $X$ is 
the sum of dependent indicator random variables of a special form. 
We show that, for large deviations, this inequality is optimal 
whenever $X$ is approximately Poisson, i.e., when the dependencies 
are weak. We also present correlation-based approaches that, in 
certain symmetric applications, yield related conclusions when $X$ 
is no longer close to Poisson. 
As an illustration we, e.g., consider subgraph counts in random 
graphs, and obtain new lower tail estimates, extending earlier work 
(for the special case $\eps=1$) of Janson, {\L}uczak and 
Ruci{\'n}ski. 
\end{abstract}

\section{Introduction}\label{sec:intro}
In probabilistic combinatorics and related areas it often is important 
to estimate the probability that a sum $X$ of dependent indicator 
random variables is small or zero (to, e.g., show that few or none 
of a collection of events occurs). Moreover, it frequently is 
desirable that these probabilities are exponentially small (to, e.g., 
make union bound arguments amenable). 
In this paper we focus on such sharp estimates for the \emph{lower tail} 
$\Pr(X \le (1-\eps)\E X)$, where $X$ is of a form that is commonly 
used in, e.g., applications of the probabilistic method or random 
graph theory, see \cite{AS, JLR}. 
More precisely, the underlying probability space is the random subset 
$\Gamma_{\vp} \subseteq \Gamma$, with $|\Gamma|=N$ and 
$\vp = (p_i)_{i \in \Gamma}$, where each $i \in \Gamma$ is included, 
independently, with probability $p_i$. Given a family 
$\bigl(Q(\alpha)\bigr)_{\alpha \in \cX}$ of subsets of 
$\Gamma$ (often $\cX \subseteq 2^\Gamma$ and
$Q(\alpha)=\alpha$ is convenient) we define 
$I_{\alpha} = \indic{Q(\alpha) \subseteq \Gamma_{\vp}}$, so that
\begin{equation}\label{def:X}
X=\sum_{\alpha \in \cX} I_{\alpha} 
\end{equation}
counts the number of sets $Q(\alpha)$ that 
are entirely contained in $\Gamma_{\vp}$. We write $\alpha \sim \beta$ 
if $Q(\alpha) \cap Q(\beta) \neq \emptyset$ and $\alpha \neq \beta$, 
which intuitively means that there are `dependencies' between 
$I_{\alpha}$ and $I_{\beta}$. Let 
\begin{gather*}
\mu = \E X = \sum_{\alpha \in \cX} \E I_{\alpha}, \qquad \Pi = \max_{\alpha \in \cX}\E I_{\alpha}, \\
\Lambda = \mu + \sum_{(\alpha,\beta) \in \cX\times\cX: \alpha \sim \beta} \E I_{\alpha}I_{\beta} = (1+\delta) \mu .
\end{gather*}
(We write $\mu(X)$, $\Pi(X)$, $\Lambda(X)$ and $\delta(X)$ in case 
of ambiguity.) Note that $\delta$ measures how dependent the 
indicators $I_{\alpha}$ are (with $\delta=0$ in the case of 
independent summands), and that $\Var X \le \Lambda$ holds.  
In~\cite{Janson} the first author proved the following lower tail 
analogue 
(often called \emph{Janson's inequality}, see, e.g.,~\cite{AS})
of the Bernstein and Chernoff bounds for sums of independent
indicators (the case $\delta=0$):
with $\varphi(x)=(1+x)\log(1+x)-x$, 
for all $\eps \in [0,1]$ we have 
\begin{equation}\label{eq:J}
\Pr(X \le (1-\eps) \E X) \le \exp\bigl\{-\varphi(-\eps) \mu/(1+\delta)\bigr\} = \exp\bigl\{-\varphi(-\eps) \mu^2/\Lambda\bigr\} ,
\end{equation}
where $\varphi(-1)=1$, $\eps^2/2 \le \varphi(-\eps) \le \eps^2$ and 
$\varphi(-\eps)=\eps^2/2+O(\eps^3)$ for $\eps \in [0,1]$. 
As discussed in~\cite{Janson,JLR,AS}, inequality \eqref{eq:J} is 
quite attractive because it (i) yields Poisson-like tail estimates 
in the weakly dependent case $\delta=O(1)$, (ii) usually corresponds 
to a (one-sided) exponential version of Chebyshev's inequality, and 
(iii) often qualitatively matches the tail behaviour suggested by 
the central limit theorem. 
For example, it is well-known (and not hard to check) that 
$\Lambda = \Theta(\Var X)$ if $\widehat{p}=\max\{\Pi,\max_{i}p_i\}$ is 
bounded away from one, that $\widehat{p} \to 0$ implies 
$\Lambda \sim \Var X$, and that $\delta,\Pi \to 0$ implies 
$\Lambda \sim \mu \sim \Var X$.

The inequality~\eqref{eq:J} is nowadays a widely used tool in 
probabilistic combinatorics (see, e.g.,~\cite{AS, JLR} and the 
references therein), which makes it important to understand how 
`sharp' it is, i.e., whether the exponential rate of decay given by 
\eqref{eq:J} is best possible. For sums of independent Bernoulli 
random variables we have $\delta=0$ and \eqref{eq:J} coincides with the 
Chernoff bounds, where the exponent is well-known to be best possible 
if $\max_i p_i=o(1)$. However, it is doubtful whether such examples
are of any significance for concrete applications with $\delta > 0$.
Fortunately, whenever $\Pi < 1$, 
Harris' inequality~\cite{Harris1960} gives,
as noted in \cite{JLR_ineq}, 
\begin{equation}\label{eq:H}
\Pr(X =0) \ge \prod_{\alpha \in \cX}(1-\E I_{\alpha}) \ge \exp\bigl\{-\mu/(1-\Pi)\bigr\} . 
\end{equation}
The point is that \eqref{eq:J} and \eqref{eq:H} yield 
$\log \Pr(X=0) \sim -\mu$ whenever $\delta,\Pi \to 0$. 
This raises the intriguing question whether the exponent of 
\eqref{eq:J} is also sharp for other choices of $\eps$, in 
particular when $\eps \to 0$ (which, of course, is also an 
interesting problem in concentration of measure).

\subsection{Main result}
In this paper we prove that ``Janson's inequality''~\eqref{eq:J} is close 
to best possible in many situations of interest. 
Our first result shows that, for large deviations, the rate of decay 
of \eqref{eq:J} is optimal for \emph{any} random variable $X$ of type 
\eqref{def:X} that is approximately Poisson, i.e., whenever 
$\delta,\Pi \to 0$ (see~\cite{Janson}). 
\begin{theorem}\label{thm:LT}
With notations as above, if $\eps \in [0,1]$, $\max\{\Pi,\indic{\eps < 1}\delta\} \le 2^{-14}$ 
and $\eps^2\mu \ge \indic{\eps < 1}$, then 
\begin{equation}\label{eq:LT}
\Pr(X \le (1-\eps) \E X) \ge \exp\bigl\{-(1+\xi) \varphi(-\eps) \mu \bigr\} , 
\end{equation}
with $\xi = 135 \max\{\Pi^{1/8},\indic{\eps < 1}\delta^{1/8},\indic{\eps < 1}(\eps^2\mu)^{-1/4}\}$. 
\end{theorem}
With $\varphi(-1)=1$ in mind, note that \eqref{eq:LT} qualitatively 
extends the lower bound \eqref{eq:H} resulting from Harris' 
inequality~\cite{Harris1960} to general $\eps$. Here the condition 
$\eps^2 \mu = \Omega(1)$ is natural in the context of exponentially 
small probabilities since $(1+\xi)\varphi(-\eps) = \Theta(\eps^2)$. 
As discussed, our favourite range is when $\delta,\Pi \to 0$. 
For large deviations, i.e., when $\eps^2 \mu \to \infty$ holds, 
\eqref{eq:J} and \eqref{eq:LT} then yield 
\begin{equation*}\label{eq:LDRF}
\log \Pr(X \le (1-\eps) \E X) \sim -\varphi(-\eps) \mu .
\end{equation*}
In words, Theorem~\ref{thm:LT} determines the \emph{large deviation 
rate function} $\log \Pr(X \le (1-\eps) \E X)$ up to second order 
error terms, closing a gap that was left open by the first author 
nearly 25 years ago. Indeed, Theorem~2 in~\cite{Janson} 
gives a lower bound, but it is at best off from the upper bound \eqref{eq:J} 
by a (multiplicative) constant factor in the exponent, and even this 
holds only for a more restricted range of the parameters.
Furthermore, Theorem~\ref{thm:LT} with $\delta=0$ also implies the  
optimality of the Chernoff bounds mentioned above.

Our second result yields a related conclusion when $\delta=O(1)$ and 
$\Pi$ is bounded away from one. More precisely, in this `weakly 
dependent' case Theorem~\ref{thm:LT2} shows that the decay of 
the inequality~\eqref{eq:J} is best possible up to constant 
factors in the exponent. 
\begin{theorem}\label{thm:LT2}
With notations as above, if $\eps \in [0,1]$, $\Pi < 1$ and 
$\eps^2\mu\ge \indic{\eps < 1/50}(1+\delta)^{-1/2}$, then 
\begin{equation}\label{eq:LT2}
\Pr(X \le (1-\eps) \E X) \ge \exp\bigl\{-K\varphi(-\eps) \mu(1+\delta^*) \bigr\} \ge \exp\bigl\{-K \eps^2 \mu (1+\delta^*)\bigr\} , 
\end{equation}
with $K=5000/(1-\Pi)^{5}$ and $\delta^*=\indic{\eps < 1/50}\delta$. 
\end{theorem}
A key feature of \eqref{eq:LT2} is that it holds for \emph{any} 
$\Pi < 1$ (and that the dependence of $K$ on $\Pi$ is explicit). 
Note that usually $K=\Theta(1)$. Whenever $\delta=O(1)$, inequalities 
\eqref{eq:J} and \eqref{eq:LT2} then yield 
\begin{equation*}\label{eq:LDRF2}
\log \Pr(X \le (1-\eps) \E X) = -\Theta(\eps^2 \mu) ,
\end{equation*}
where the implicit constants differ by a factor of at most 
$2K(1+\delta)^2=O(1)$. This subsumes the folklore fact that Chernoff 
bounds (where $\delta=0$) are sharp up to constants in the exponent 
if $\max_i p_i$ is bounded away from one. 
While the numerical value of $K$ is often immaterial, better constant 
factors can typically be obtained, if desired, by reworking the proof 
(optimizing certain parameters to the situation at hand).

The proofs of Theorem~\ref{thm:LT} and~\ref{thm:LT2} hinge on 
H{\" o}lder's inequality and several estimates of the Laplace transform 
(which in turn are based on correlation inequalities), see 
Section~\ref{sec:LT}. 
In fact, an inspection of the proofs reveals that Theorem~\ref{thm:LT} 
and~\ref{thm:LT2} (as well as \eqref{eq:H}, Theorem~\ref{thm:LT3} and 
Lemma~\ref{lem:LT4}) remain valid for the more general correlation 
conditions (and setup) stated by Riordan and Warnke~\cite{RiordanWarnke2012J}. 
It would be interesting to know whether similar results also hold under the 
weaker dependency assumptions of Suen's inequality \cite{Suen,Janson:Suen}.

\subsection{Main example}
From an applications point of view it is important to also understand 
the sharpness of \eqref{eq:J} in the case $\delta=\Omega(1)$, i.e., 
when $X$ is no longer close to Poisson. In Section~\ref{sec:bstrp} we 
present correlation-inequality based bootstrapping approaches which 
often allow us to deal with this remaining `strongly dependent' case. 
The punchline seems to be that, in the presence of certain symmetries, 
the inequality~\eqref{eq:J} is oftentimes best possible up to 
constant factors in the exponent.

In this paper our main example is the number of small subgraphs in the 
binomial random graph $G_{n,p}$, which is a classical topic in random 
graph theory (see, e.g.,~\cite{ER1960,BB1981,R88}). 
It frequently serves as a test-bed for new probabilistic estimates (see, 
e.g.,~\cite{B82,JLR_ineq,Spencer1990,KimVu2000,DL,UTSG,CV2011}), and we 
shall use it to demonstrate the applicability of our bootstrapping approaches. 
In fact, we consider the more general random hypergraph $G^{(k)}_{n,p}$, 
with $k \ge 2$, where each of the $\binom{n}{k}$ edges of the complete $
k$-uniform hypergraph $K_n^{(k)}$ is included, independently, with 
probability $p$. Given a $k$-uniform hypergraph $H$, or briefly 
\emph{$k$-graph}, we define $X_H=X_H(n,p)$ as the number of copies of 
$H$ in $G^{(k)}_{n,p}$, where by a \emph{copy} we mean, as usual, a 
subgraph isomorphic to $H$. Furthermore, we write $e_H=|E(H)|$ and 
$v_H=|V(H)|$ for the number of edges and vertices of $H$, respectively. 
Theorem~\ref{thm:HG:const} shows that the lower tail of the distribution 
of $X_H$ is governed by $\Phi_H$, i.e., the expected number of copies 
of the `least expected' subgraph of $H$. This exponential rate of decay 
is consistent with normal approximation heuristics since 
$\Phi_H=\Theta\bigl((1-p)(\E X_H)^2/\Var X_H\bigr)$, see Lemma~3.5 in~\cite{JLR}. 
\begin{theorem}\label{thm:HG:const}
Let $H$ be a $k$-graph with $e_H \ge 1$. Define $\Phi_H = \Phi_H(n,p) = \min \{ \E X_J: J \subseteq H, e_J \ge 1\}$. 
There are positive constants $c$, $C$, $D$ and $n_0$, all depending 
only on $H$, such that for all $n \ge n_0$, $p \in [0,1)$ and 
$\eps \in [0,1]$ satisfying $\eps^2 \Phi_H \ge \indic{\eps < 1}D$ we have
\begin{equation}\label{eq:HG:const}
\exp\bigl\{- (1-p)^{-5} C \eps^2 \Phi_H\bigr\} \le \Pr(X_{H} \le (1-\eps) \E X_{H}) \le \exp\bigl\{-c \eps^2 \Phi_H\bigr\}.
\end{equation}
\end{theorem}
The upper bound of \eqref{eq:HG:const} follows from \eqref{eq:J} via 
standard calculations (see, e.g., \cite{JLR} or Lemma~\ref{lem:lambda:HG}), 
and so the real content of this theorem is the `matching' lower bound. 
A key feature of Theorem~\ref{thm:HG:const} is that $\eps$ is 
\emph{not} fixed, but may depend on $n$. In the context of exponentially 
decaying probabilities, note that the $\eps^2 \Phi_H = \Omega(1)$ condition 
is natural (unless $p \approx 1$). In applications $p$ is typically 
bounded away from one (in fact, $p=o(1)$ is often standard), in which 
case \eqref{eq:HG:const} yields 
\begin{equation}\label{eq:HG:LDRF:const}
\log \Pr(X_{H} \le (1-\eps) \E X_{H}) = - \Theta(\eps^2 \Phi_H) ,
\end{equation}
determining the large deviation rate function of $X_H$ up to constants 
factors. For the special case $\eps=1$ (and $k=2$) this was established 
more than 25 years ago by Janson, {\L}uczak and Ruci{\'n}ski~\cite{JLR_ineq}, 
and for $\eps \ge \eps_0$ an analogous statement is nowadays easily 
deduced from \eqref{eq:J} and \eqref{eq:H}, see also \eqref{eq:HG:const:H}. 
By contrast, the case $\eps \to 0$ seems to have eluded further 
attention, and Theorem~\ref{thm:HG:const} rectifies this (surprising) 
gap in the literature.

Although not our primary focus, in certain ranges our proof techniques 
are strong enough to establish the finer behaviour of the large 
deviation rate function. In particular, for the case in which there is 
only one subgraph $G \subseteq H$ with $\E X_G = \Theta(\Phi_H)$ we 
have two results that determine the leading constant in \eqref{eq:HG:LDRF:const}. 
More precisely, Theorem~\ref{thm:HG:ULO} applies if there is only one 
copy of $G$ in $H$ (which includes the case $G=H$), and Theorem~\ref{thm:HG:extr} 
applies if $G$ is an edge (in which case there are $e_H$ copies of 
$G$ in $H$). To state these results, for any given $k$-graph $H$ we set 
\begin{equation}\label{mk}
m_k(H) = \indic{e_H \ge 2}\max_{J \subseteq H, e_J \ge 2} \frac{e_J-1}{v_J-k} + \indic{e_H=1}\frac{1}{k}.   
\end{equation}
In addition, we define $\ex(n,H)$ as the maximum number of edges in 
an $H$-free $k$-graph with $n$ vertices. It is well-known (see, e.g.,~\cite{KHG})
that $\pi_H=\lim_{n \to \infty}\ex(n,H)/\binom{n}{k}$ exists, with 
$\pi_H \in [0,1)$, and that for graphs (i.e., $k=2$) we have 
$\pi_H=1-1/(\chi(H)-1)$, where $\chi(H)$ is the chromatic number 
of $H$. 
\begin{theorem}\label{thm:HG:ULO}
Let $G \subseteq H$ be $k$-graphs with $e_G \ge 1$. 
Assume that there is exactly one copy of $G$ in $H$, and that 
$p=p(n)=o(1)$ is such that $\E X_G = o(\E X_J)$ for all $G \neq J \subseteq H$ with $e_J \ge 1$. 
If $\eps=\eps(n) \in (0,1]$ satisfies $\eps^2 \E X_G \ge\indic{\eps<1}\omega\bigl(1+\indic{G \neq H, e_G \ge 2}\log(1/\eps)\bigr)$, then we have
\begin{equation}\label{eq:HG:ULO}
\log \Pr(X_{H} \le (1-\eps) \E X_{H}) \sim -\varphi(-\eps) \E X_{G} .
\end{equation}
\end{theorem}
\begin{theorem}\label{thm:HG:extr}
Let $H$ be a $k$-graph with $e_H \ge 1$. 
If $p=p(n) =o(1)$ and $\eps=\eps(n) \in [0,1]$
satisfy $p = \omega(n^{-1/m_k(H)})$ and $\eps^2 \binom{n}{k}p = \omega(1)$, then we have 
\begin{equation}\label{eq:HG:extr}
\log \Pr(X_{H} \le (1-\eps) \E X_{H}) \sim 
\begin{cases}
		-\varphi(-\eps)\binom{n}{k}p/e_H^2 , & \text{if $\eps=o(1)$,}\\
		-\varphi(-\eps)\binom{n}{k}p (1-\pi_H) , & \text{if $\eps=1-o(1)$.}\\
	\end{cases}
\end{equation}
\end{theorem}
Here our main contributions are the tight lower bound of \eqref{eq:HG:ULO}, 
and the case $\eps=o(1)$ of \eqref{eq:HG:extr}. 
Theorem~\ref{thm:HG:ULO} is a natural extension of earlier work of 
Janson, {\L}uczak and Ruci{\'n}ski~\cite{JLR_ineq} for the special case 
$\eps=1$ (and $k=2$). 
Theorem~\ref{thm:HG:extr} partially solves an open problem of~\cite{JLR_ineq}, 
but in the relevant case $\eps=1$ inequality \eqref{eq:HG:extr} is a 
fairly simple consequence of the recent `hypergraph container' 
results of Saxton and Thomason~\cite{ST}, see also Lemma~\ref{lem:G:zero}. 
With $\varphi(-\eps) = \Theta(\eps^2)$ in mind the conditions involving 
$\eps^2$ are natural in both results -- up to the logarithmic term in 
case of Theorem~\ref{thm:HG:ULO}, which seems to be an artefact of 
our proof (we leave its removal as an open problem, see 
Section~\ref{sec:bstrp:sdcmp}). 
The form of the exponent in Theorem~\ref{thm:HG:extr} differs in an 
intriguing way for $\eps = o(1)$ and $\eps=1-o(1)$. In particular, 
\eqref{eq:HG:extr} provides a natural example where the  
inequality~\eqref{eq:J} does \emph{not} always give the correct 
constants in the exponent when $\delta = \omega(1)$: 
in the case $\eps=1-o(1)$, the `extremal' structural properties of 
$H$-free graphs come into play. We leave it as an open problem to 
determine the finer behaviour of the exponent (i.e., with explicit 
constants) in the `intermediate' range $\eps=\Theta(1)$. 
This seems of particular interest since Theorem~\ref{thm:HG:ULO} 
and~\ref{thm:HG:extr} nearly cover all edge probabilities~$p$ for 
\emph{balanced} $k$-graphs with $e_H \ge 2$ and $m_k(H)=(e_H-1)/(v_H-k)$, 
where $G=H$ for $p = o(n^{-1/m_k(H)})$;
for $k=2$ (when this class usually is called \emph{2-balanced})
this class includes, e.g., 
trees, cycles, complete graphs, complete $r$-partite graphs 
$K_{t,\ldots,t}$ and the $d$-dimensional cube.

Finally, Theorems~\ref{thm:HG:const}--\ref{thm:HG:extr} compare 
favourable with related work for the \emph{upper tail} probability 
$\Pr(X_H \ge (1+\eps) \E X_H)$, where the case $\eps=\Theta(1)$ has 
been extensively studied for $k=2$, see, e.g., \cite{Spencer1990,Vu2001,UTSG,K3TailCh,KkTailDK,MS,CD2014} 
and the references therein. 
Indeed, for most graphs $H$ the order of magnitude of the large 
deviation rate function $\log \Pr(X_H \ge (1+\eps) \E X_H)$ is only 
known up to logarithmic factors when $\eps=\Theta(1)$, whereas 
Theorem~\ref{thm:HG:const} determines $\log \Pr(X_H \le (1-\eps) \E X_H)$ 
up to constant factors, even when $\eps=\eps(n) \to 0$. 
For triangles the finer behaviour of 
$\log \Pr(X_{K_3} \ge (1+\eps) \E X_{K_3})$ 
has very recently been determined for $\eps=\Theta(1)$ and 
$n^{-1/42+o(1)} \le p=o(1)$, see~\cite{LZ2014}.
By contrast, for all balanced $k$-graphs $H$ (which for $k=2$ includes 
$H=K_3$) Theorems~\ref{thm:HG:ULO}--\ref{thm:HG:extr} apply for 
essentially all $p=o(1)$ of interest, excluding only $p=\Theta(n^{-1/m_k(H)})$. 
However, the key conceptual difference is that Theorem~\ref{thm:HG:ULO} 
includes the case $\eps=\eps(n)\to 0$.

The rest of the paper is organized as follows. 
First, in Section~\ref{sec:LT}, we prove Theorem~\ref{thm:LT} 
and~\ref{thm:LT2}. Next, in Section~\ref{sec:bstrp}, we present several 
bootstrapping approaches that yield lower bounds for the lower tail, 
which are subsequently illustrated in Section~\ref{sec:ex}. 
Namely, in Section~\ref{sec:RS} we apply them to the number of 
arithmetic progressions in random subsets of the integers, and in 
Section~\ref{sec:RHG} we apply them to subgraph counts in random 
hypergraphs and prove Theorems~\ref{thm:HG:const}--\ref{thm:HG:extr}.

\section{Lower bounds for the lower tail}\label{sec:LT}
In this section we prove Theorem~\ref{thm:LT} and~\ref{thm:LT2}, i.e., 
establish lower bounds for the lower tail. Since our core argument 
breaks down when $\eps$ is very close to one, en route to 
Theorem~\ref{thm:LT} we establish the following (slightly sharper) 
complementary estimates. 
\begin{theorem}\label{thm:LT3}
Let $X=\sum_{\alpha \in \cX} I_{\alpha}$, $\mu=\E X$, $\Pi$ and 
$\delta$ be defined as in Section~\ref{sec:intro}. 
If $e(1-\eps)\eps^2\mu \ge 1$ and $0 \le \eps \le 1-4\max\{\Pi^{1/4},\delta^{1/4}\}$, then 
\begin{equation}\label{eq:LT3}
\Pr(X < (1-\eps) \E X) \ge \exp\bigl\{-(1+\xi) \varphi(-\eps) \mu \bigr\} , 
\end{equation}
with $\xi = 135 \max\{\Pi^{1/4},\delta^{1/4},[e(1-\eps)\eps^2\mu]^{-1/2}\}$. 
\end{theorem}
\begin{lemma}\label{lem:LT4}
Let $X=\sum_{\alpha \in \cX} I_{\alpha}$, $\mu=\E X$ and $\Pi$ be 
defined as in Section~\ref{sec:intro}. 
If $1-e^{-1} \le \eps \le 1$ and $\Pi < 1$, then 
\begin{equation}\label{eq:LT4}
\Pr(X \le (1-\eps)\E X) \ge \Pr(X = 0) \ge \exp\bigl\{-(1+\zeta) \varphi(-\eps) \mu \bigr\} ,
\end{equation}
with $\zeta = 10 \max\{\sqrt{1-\eps},\Pi/(1-\Pi)\}$. 
\end{lemma}
While Lemma~\ref{lem:LT4} follows from~\eqref{eq:H} via calculus 
(see Lemma~\ref{lem:varphi2}), the remaining proofs are not a mere 
refinement of~\cite{Janson}, but contain several new ideas and 
ingredients. This includes integrating the logarithmic derivative of 
the Laplace transform over the interval $[r,t]$ instead of the usual 
$[0,t]$ (see the proof of Lemma~\ref{lem:JU}), using H{\" o}lder's 
inequality with parameter $p \to 1$ instead of the Cauchy--Schwarz 
inequality (see Section~\ref{sec:prf:str}), and a careful treatment 
of second order error terms (see, e.g., Lemma~\ref{lem:JL} and~\ref{lem:negl}).

\subsection{Preliminaries}
We first collect some basic estimates of the Laplace transform of~$X$ 
as defined in Section~\ref{sec:intro}. 
\begin{lemma}\label{lem:JL}
For all $s \ge 0$ satisfying $\lambda = \Pi (1-e^{-s}) < 1$ we have
\begin{equation}\label{eq:JL}
\E e^{-sX} \ge \exp\left\{-\mu (1-e^{-s})-\frac{\mu \Pi (1-e^{-s})^2}{2(1-\lambda)}\right\} . 
\end{equation}
\end{lemma}
\begin{proof}
The FKG inequality~\cite{FKG} (or Harris's inequality \cite{Harris1960})
yields  
\[
\E e^{-sX} = \E \prod_{\alpha \in \cX }e^{-sI_{\alpha}} \ge \prod_{\alpha \in \cX } \E e^{-sI_{\alpha}} = \prod_{\alpha \in \cX } \bigl(1-\E I_{\alpha} (1-e^{-s})\bigr) .
\]
Now, for $x \in [0,1)$ we have 
\begin{equation}\label{eq:taylor:log}
\log(1-x)=-\sum_{j \ge 1} \frac{x^j}{j} \ge -x-\frac{x^2}{2(1-x)}, 
\end{equation}
and \eqref{eq:JL} follows since $\E I_{\alpha} \le \Pi$ and $\mu = \sum_{\alpha \in \cX} \E I_{\alpha}$. 
\end{proof}
\begin{lemma}\label{lem:JU}
For all $t \ge r \ge 0$ we have 
\begin{equation}\label{eq:JU}
\frac{\E e^{-rX}}{\E e^{-tX}} \ge \exp\left\{\frac{\mu}{1+\delta}\left(e^{-(1+\delta)r}-e^{-(1+\delta)t}\right)\right\} .
\end{equation}
\end{lemma}
\begin{proof}
Let $\Psi(x)= \E e^{-xX}$. The proof of Lemma~1 in~\cite{Janson} 
establishes $-\frac{d}{dx} \log \Psi(x) \ge \mu e^{-(1+\delta)x}$ for 
$x \ge 0$ (see also~\cite{RiordanWarnke2012J}). 
Hence 
\begin{equation*}\label{eq:JU:int}
\begin{split}
\log\left(\frac{\E e^{-rX}}{\E e^{-tX}}\right) & = -\log \Psi(t) + \log \Psi(r) = \int_{r}^{t} \left(-\frac{d}{dx} \log \Psi(x)\right) dx \\
&\ge \int_{r}^{t} \mu e^{-(1+\delta)x} dx = \frac{\mu}{1+\delta}\left(e^{-(1+\delta)r}-e^{-(1+\delta)t}\right) , 
\end{split}
\end{equation*}
and \eqref{eq:JU} follows. 
\end{proof}

Next, we state some technical estimates of $\varphi(-\eps)=(1-\eps)\log(1-\eps)+\eps$ 
for later reference (these can safely be skipped on first reading). 
Following standard conventions, for $k \in \{1,2\}$ we have 
$0\log^k(0)=\lim_{\eps \nearrow 1}(1-\eps)\log^k(1-\eps)=0$, so that $\varphi(-1)=1$. 
\begin{lemma}\label{lem:varphi}
For all $\eps \in [0,1]$ we have 
\begin{equation}\label{eq:varphi}
\max\bigl\{(1-\eps) \log^2(1-\eps),\eps^2\bigr\} \le 2 \varphi(-\eps) \le \min\bigl\{\log^2(1-\eps),2\eps^2\bigr\} .
\end{equation}
\end{lemma}
\begin{lemma}\label{lem:varphi2}
For all $1-e^{-1} \le \eps \le 1$ we have 
\begin{equation}\label{eq:varphi2}
\varphi(-\eps) \le 1 \le (1+5 \sqrt{1-\eps}) \varphi(-\eps) .
\end{equation} 
\end{lemma}
\begin{lemma}\label{lem:varphi3}
For all $\eps \in [0,1]$ and $A \in [0,\infty)$ we have, with $\gamma=A-1$, 
\begin{equation}\label{eq:varphi3}
\varphi(-A\eps) \le \begin{cases}
		(1+A\eps)A^2 \varphi(-\eps), & \text{if $A\eps \le 1$,}\\
		(1+\sqrt{\gamma}) \varphi(-\eps), & \text{if $0 \le 3\sqrt{\gamma} \le 1-\eps$.}
	\end{cases}
\end{equation} 
\end{lemma}
The elementary proofs of Lemma~\ref{lem:varphi}--\ref{lem:varphi3} are 
deferred to Appendix~\ref{apx}.

\subsection{Proof strategy}\label{sec:prf:str}
We start with a general lower bound for $\Pr(X < (1-\eps) \E X)$. 
If $p,q \in (1,\infty)$ satisfy $1/p+1/q=1$, then H{\" o}lder's inequality implies 
\[
\E(e^{-sX} \indic{X < (1-\eps) \E X}) \le (\E e^{-psX})^{1/p} \Pr(X < (1-\eps) \E X)^{1/q} .
\] 
Noting that $q=q/p+1=1/(p-1)+1$, we infer 
\begin{equation}\label{eq:LTH}
\begin{split}
\Pr(X < (1-\eps) \E X) & \ge \left(\frac{\E(e^{-sX} \indic{X < (1-\eps) \E X})}{ (\E e^{-psX})^{1/p}}\right)^{q} \\
& = \left(\frac{\E(e^{-sX} \indic{X < (1-\eps)\E X})}{\E e^{-sX} }\right)^{\frac{p}{p-1}} \cdot \left(\frac{\E e^{-sX} }{\E e^{-psX} }\right)^{\frac{1}{p-1}} \E e^{-sX} .
\end{split}
\end{equation}

In the following we heuristically outline how we estimate 
$\Pr(X < (1-\eps) \E X)$ when $\delta,\Pi \to 0$ and $\eps<1$ (to be 
precise, $\eps$ bounded away from one). The idea is to first consider 
$p>1$ and $s > z=-\log(1-\eps)$, and then let $p \to 1$ and $s \to z$. 
Since $\Pi \to 0$, using Lemma~\ref{lem:JL} we have 
\begin{equation}\label{eq:heur:JL}
\E e^{-sX} \ge \exp\Bigl\{-\mu\bigl(1-e^{-s}+o(1)\bigr)\Bigr\} .
\end{equation}
So, using Lemma~\ref{lem:JU} together with $\delta \to 0$, we expect 
that (replacing the difference quotient by the derivative), as $p \to 1$, 
\begin{equation}\label{eq:heur:JU}
\begin{split}
\left(\frac{\E e^{-sX} }{\E e^{-psX} }\right)^{\frac{1}{p-1}} & \ge \exp\Bigl\{\mu s \left(\frac{e^{-(1+\delta)s}-e^{-(1+\delta)ps}}{(1+\delta)(p-1)s}\right)\Bigr\} 
= \exp\Bigl\{\mu \bigl(s e^{-s}+o(1)\bigr)\Bigr\} .
\end{split}
\end{equation}
The point is that $1-e^{-s}-se^{-s} \to \varphi(-\eps)$ as $s \to z$. 
So, if \eqref{eq:heur:JL} and \eqref{eq:heur:JU} essentially 
determine the right hand side of \eqref{eq:LTH}, then our previous 
considerations suggest 
\[
\Pr(X < (1-\eps) \E X) \ge \exp\Bigl\{-\mu \bigl(\varphi(-\eps) + o(1)\bigr) \Bigr\} .
\]
Luckily, our later calculations confirm that (for suitable choices 
of $p$ and $s$) we can indeed essentially ignore the first term on 
the right hand side of \eqref{eq:LTH} for large deviations, i.e., 
when $\eps^2 \mu \to \infty$ holds.

\subsection{Proofs of Theorem~\ref{thm:LT2} and~\ref{thm:LT3}}
Assume that $\eps, \tau \in (0,1)$ and $\sigma \in (0,\infty)$. Let
\begin{equation}\label{def:pq}
p = 1+\sigma \quad \text{ and } \quad q=1+1/\sigma,
\end{equation}
so that $p,q \in (1,\infty)$ and $1/p+1/q=1$. Furthermore, let 
\begin{equation}\label{def:zs}
z=-\log(1-\eps) \quad \text{and} \quad s = p z .
\end{equation}
With \eqref{eq:LTH} in mind, the following two lemmas are at the 
heart of our argument. 
\begin{lemma}\label{lem:lower}
With definitions as above, if $\Pi (1-e^{-s}) \le 1/2$, then 
\begin{equation}\label{eq:lower}
\left(\frac{\E e^{-sX} }{\E e^{-psX} }\right)^{\frac{1}{p-1}} \E e^{-sX} 
\ge e^{-(1+\eta) \varphi(-\eps) \mu} ,
\end{equation}
with $\eta = 2p^2(\sigma+p\delta + \Pi)+2p\sigma$. 
\end{lemma}
\begin{proof}
Since $f(x)=-e^{-x}$ satisfies $f'(x)=e^{-x}$, the mean value theorem 
implies that there is $\zeta \in [1,p]$ such that 
\begin{equation}\label{eq:lower:mvt}
\frac{e^{-(1+\delta)s}-e^{-(1+\delta)ps}}{(1+\delta)(p-1)s} = e^{-(1+\delta)\zeta s} \ge e^{-(1+\delta)p s} .
\end{equation} 
Furthermore, since $g(x)=e^{-x}$ satisfies $g'(x)=-e^{-x}$ and 
$g''(x)=e^{-x} \ge 0$, using Taylor's theorem with remainder, we obtain 
\begin{equation}\label{eq:lower:mvt2}
e^{-(1+\delta)p s} \ge e^{-s} - \bigl((1+\delta)p -1\bigr)se^{-s} .
\end{equation}
Note that $(1+\delta)p -1= \sigma + p\delta$. 
Furthermore, since $s=-p\log(1-\eps)$, Bernoulli's inequality yields 
\begin{equation}\label{eq:BI}
(1-e^{-s})^2 =(1-(1-\eps)^{p})^2 \le p^2\eps^2.
\end{equation}
So, by 
combining Lemmas~\ref{lem:JL} and~\ref{lem:JU} with
\eqref{eq:lower:mvt}--\eqref{eq:BI},  
using $\Pi (1-e^{-s}) \le 1/2$, it follows that 
\begin{equation*}\label{eq:lower:lb}
\begin{split}
\left(\frac{\E e^{-sX} }{\E e^{-psX} }\right)^{\frac{1}{p-1}} \E e^{-sX} 
& \ge \exp\left\{\frac{\mu s\bigl(e^{-(1+\delta)s}-e^{-(1+\delta)ps}\bigr)}{(1+\delta)(p-1)s} - \mu(1-e^{-s}) - \mu \Pi(1-e^{-s})^2\right\}\\
& \ge \exp\left\{-\mu\Bigl(1-e^{-s}-s e^{-s}+\bigl(\sigma + p\delta\bigr) s^2 e^{-s} + \Pi p^2\eps^2 \Bigr) \right\} . 
\end{split}
\end{equation*}
Let $g(x)=1-e^{-x}-x e^{-x}$, and note that $g(z) = \varphi(-\eps)$. 
Furthermore, for $z \le x \le s$ we have $g'(x) = xe^{-x} \le se^{-z}$. 
So, using Taylor's theorem with remainder, we deduce that
\begin{equation*}
1-e^{-s}-s e^{-s} \le \varphi(-\eps) + (s-z) s e^{-z}.
\end{equation*}
Consequently, since $s=pz \ge z$, we obtain 
\begin{equation*}
\begin{split}
\left(\frac{\E e^{-sX} }{\E e^{-psX} }\right)^{\frac{1}{p-1}} \E e^{-sX} 
& \ge \exp\Bigl\{-\varphi(-\eps)\mu - \bigl(z^2e^{-z} \eta_1 + \eps^2 \eta_2 \bigr)\mu\Bigr\} ,
\end{split}
\end{equation*}
where $\eta_1 = p^2(\sigma+p\delta)+p\sigma$ and $\eta_2 = p^2\Pi$. 
Finally, recalling $z=-\log(1-\eps)$, the point is that 
Lemma~\ref{lem:varphi} yields $\max\{z^2e^{-z},\eps^2\} \le 2 \varphi(-\eps)$,
yielding the result with $\eta=2\eta_1+2\eta_2$. 
\end{proof}

\begin{lemma}\label{lem:negl}
With definitions as above, if $\lambda=\Pi (1-e^{-s}) <1$ and 
$(1-\tau) \sigma^2(1-\eps)^{p} \ge p^2\Pi/(1-\lambda)+\delta/(1+\delta)$, 
then 
\begin{equation}\label{eq:negl}
\left(\frac{\E(e^{-sX} \indic{X < (1-\eps)\E X})}{\E e^{-sX} }\right)^{\frac{p}{p-1}} \ge \exp\left\{-\left(\frac{4p}{\tau\sigma^3(1-\eps)^p\eps^4\mu^2}\right) \varphi(-\eps) \mu \right\} . 
\end{equation}
\end{lemma}
\begin{proof}
As $p=1+\sigma$, we write 
\begin{equation}\label{eq:negl:C}
\left(\frac{\E(e^{-sX} \indic{X < (1-\eps)\mu})}{\E e^{-sX}}\right)^{\frac{p}{p-1}} = \left(1- \frac{\E(e^{-sX} \indic{X \ge (1-\eps)\mu})}{\E e^{-sX}}\right)^{\frac{p}{\sigma}} .
\end{equation}
Let $t=z/(1+\delta)$. Recalling $\varphi(-\eps)=(1-\eps)\log(1-\eps)+\eps$, 
note that 
\[
t (1-\eps)\mu -\frac{\mu}{1+\delta}\left(1-e^{-(1+\delta)t}\right)= -\frac{\varphi(-\eps)\mu}{1+\delta} .
\] 
So, using $t\le s$ and Lemma~\ref{lem:JU} (with $r=0$), it follows that 
\begin{equation}\label{eq:negl:EU}
\begin{split}
\E(e^{-sX} \indic{X \ge (1-\eps)\mu}) &\le e^{-(s-t)(1-\eps) \mu} \cdot \E e^{-tX} 
\le \exp\left\{-s(1-\eps)\mu - \frac{\varphi(-\eps)\mu}{1+\delta}\right\} . 
\end{split}
\end{equation}
Set $h(x)=(1-\eps)x-(1-e^{-x})$, and note that $h(z) = -\varphi(-\eps)$ 
and $h'(z) = 0$. Furthermore, for $x \le s$ we have $h''(x) = e^{-x} \ge e^{-s}$. 
So, using Taylor's theorem with remainder, we obtain 
\begin{equation}\label{eq:negl:T}
(1-\eps)s -(1-e^{-s}) \ge -\varphi(-\eps) +(s-z)^2 e^{-s}/2 .
\end{equation}
Recalling $p=1+\sigma$, $s=pz$ and $\lambda=\Pi (1-e^{-s})$, by 
combining Lemma~\ref{lem:JL} with \eqref{eq:negl:EU}, \eqref{eq:negl:T} 
and $(1-e^{-s})^2 \le s^2$, we infer 
\begin{equation*}
\begin{split}
\frac{\E(e^{-sX} \indic{X \ge (1-\eps)\mu})}{\E e^{-sX} } 
&\le \exp\left\{-\mu \left((1-\eps)s -(1-e^{-s})+ \frac{\varphi(-\eps)}{1+\delta}-\frac{\Pi s^2}{2(1-\lambda)} \right)\right\} \\
& \le \exp\left\{-\mu \left( \frac{\sigma^2 (1-\eps)^pz^2}{2} -\frac{\Pi p^2z^2}{2(1-\lambda)}- \frac{\delta \varphi(-\eps)}{1+\delta}\right)\right\} .
\end{split}
\end{equation*}
Since Lemma~\ref{lem:varphi} gives $\varphi(-\eps) \le \log^2(1-\eps)/2 = z^2/2$, 
we have, by assumption, 
\begin{equation}\label{eq:negl:EU2}
\frac{\E(e^{-sX} \indic{X \ge (1-\eps)\mu})}{\E e^{-sX}} \le \exp\Bigl\{-\tau \sigma^2 (1-\eps)^pz^2\mu/2\Bigr\}. 
\end{equation}
Now, inserting \eqref{eq:negl:EU2} into \eqref{eq:negl:C}, using the 
fact that $e^{-x}+e^{-1/x} \le 1$ for $x >0$ (as in the proof of 
Theorem~2 in~\cite{Janson}), we obtain 
\begin{equation*}\label{eq:negl:EL}
\left(\frac{\E(e^{-sX} \indic{X < (1-\eps)\mu})}{\E e^{-sX} }\right)^{\frac{p}{p-1}} \ge \exp\left\{-\frac{2p}{\tau\sigma^3 (1-\eps)^pz^2\mu}\right\} . 
\end{equation*}
Finally, recalling $z=-\log(1-\eps)$, Lemma~\ref{lem:varphi} yields 
$z^2 \ge \eps^2$ and $1 \le 2\varphi(-\eps)/\eps^2$. 
\end{proof}

Combining \eqref{eq:LTH} with Lemma~\ref{lem:lower} and~\ref{lem:negl}, 
the proofs of Theorem~\ref{thm:LT2} and~\ref{thm:LT3} reduce to 
defining suitable parameters $\sigma$ and $\tau$ (our choices are 
somewhat ad-hoc, and yield fairly transparent error-terms). 
\begin{proof}[Proof of Theorem~\ref{thm:LT3}]
With foresight, let $\tau = 5/8$ and 
\begin{equation}\label{gs}
\sigma = \max\bigl\{\Pi^{1/4},\delta^{1/4},[e(1-\eps)\eps^2\mu]^{-1/2}\bigr\} . 
\end{equation}
Note that the assumption 
$0 \le \eps \le 1-4\max\{\Pi^{1/4},\delta^{1/4}\}$ implies 
$\max\{\Pi,\delta\} \le 4^{-4}$, so that 
$\lambda = \Pi (1-e^{-s}) \le \Pi \le 1/5$. 
Hence, using $e(1-\eps)\eps^2\mu \ge 1$, we see that $\sigma \le 1$ and thus 
 $p \le 2$. 
Consequently, by \eqref{gs}, we have 
\begin{equation}\label{eq:thm:LT:cond}
\sigma^4 (1-\eps)^p \eps^4 \mu^2 \ge \sigma^4 (1-\eps)^2 \eps^4 \mu^2 \ge e^{-2} 
\end{equation}
and $\sigma^2 \ge \max\{\Pi^{1/2},\delta^{1/2}\}$. In addition, by 
assumption, we have $(1-\eps)^p \ge (1-\eps)^2 \ge 16 \max\{\Pi^{1/2},\delta^{1/2}\}$. 
Since $16(1-\tau) = 6$ and $p^2/(1-\lambda) \le 5$, it follows that 
\[
(1-\tau)\sigma^2(1-\eps)^p \ge 6 \max\{\Pi,\delta\} \ge p^2 \Pi/(1-\lambda) + \delta/(1+\delta) .
\]
Now, combining \eqref{eq:LTH} with Lemmas~\ref{lem:lower}--\ref{lem:negl} 
and~\eqref{eq:thm:LT:cond}, we obtain 
\[
\Pr(X < (1-\eps) \mu) \ge e^{-(1+\kappa) \varphi(-\eps) \mu} ,
\] 
with $\kappa = 2p^2(\sigma+p\delta + \Pi)+2p\sigma + 4e^{2}\tau^{-1}p\sigma$. 
Finally, using $\sigma \ge \sigma^4\ge \max\{\delta,\Pi\}$, $p \le 2$ and
$\tau = 5/8$,  
we see that $\kappa \le 135 \sigma$. 
\end{proof}
\begin{proof}[Proof of Theorem~\ref{thm:LT2}]
Let $\tau = (1-\Pi)/5$, so that, by assumption, $\tau \in (0,1/5]$. 
The proof distinguishes two cases, which eventually establish 
\eqref{eq:LT2} by noting that Lemma~\ref{lem:varphi} gives 
$\varphi(-\eps) \le \eps^2$. 

First, we assume $0 \le \eps < \tau^2/2$. Note that then, by assumption, 
we have $0 < \eps < 1/50$ and $\delta=\delta^*$. Let $p = 2/\tau$ and 
$\sigma=p-1$. Analogous to \eqref{eq:BI} we have 
$1-e^{-s} =1-(1-\eps)^p\le p \eps$,  
so that $\Pi \le 1$ implies 
\[
\lambda=\Pi (1-e^{-s}) \le \Pi p \eps \le \tau ,
\]
which in particular yields $\lambda \le 1/2$, with room to spare. 
Next observe that, since $\sigma/p = 1-1/p$ and 
$\max\{2/p, p\eps ,\lambda\} = \tau$, by the definition of $\tau$ we 
have 
\begin{equation*}
\begin{split}
\frac{(1-\tau)\sigma^2(1-\eps)^p(1-\lambda)}{p^2} - \frac{1}{p^2} & \ge (1-\tau)(1-2/p)(1-p\eps)(1-\lambda)-\tau^2/4 \\
& \ge (1-\tau)^4 - \tau^2/4 \ge 1-5 \tau = \Pi ,
\end{split}
\end{equation*}
which in turn readily yields $(1-\tau) \sigma^2(1-\eps)^{p} \ge p^2\Pi/(1-\lambda)+\delta/(1+\delta)$. 
Similarly, using $\sigma \ge p/2=\tau^{-1}$ and $\tau \le 1/2$ we 
obtain 
\[
\tau\sigma^3(1-\eps)^{p} \ge \tau^{-2}(1-\tau) \ge \tau^{-2}/2 .
\]
Since $\eps^4\mu^2 \ge (1+\delta)^{-1}$ by assumption, analogously to 
the proof of Theorem~\ref{thm:LT3}, using \eqref{eq:LTH} together 
with Lemmas~\ref{lem:lower}--\ref{lem:negl}, we obtain 
\[
\Pr(X \le (1-\eps) \mu) \ge \Pr(X < (1-\eps) \mu) \ge e^{-(1+\kappa) \varphi(-\eps) \mu}, 
\]
with $\kappa = 2p^2(\sigma+p\delta + \Pi)+2p\sigma + 8\tau^2p(1+\delta)$. 
Now, using $\max\{\Pi,\tau\} \le 1$ and  
$\sigma \le p = 2/\tau=10/(1-\Pi)$, 
a short calculation shows that, say, 
\[
1+\kappa \le 17+2p^3+4p^2 + (2p^3+16)\delta \le 2500(1+\delta)/(1-\Pi)^3. 
\]

Finally, we assume $\tau^2/2 \le \eps \le 1$. Using the lower 
bound~\eqref{eq:H} resulting from Harris' inequality~\cite{Harris1960}, 
it follows that 
\begin{equation}\label{eq:harris:LB}
\Pr(X \le (1-\eps) \mu) \ge \Pr(X =0) \ge e^{-\mu/(1-\Pi)} . 
\end{equation}
The point is that, by assumption, we have $2/\eps^2 \le 8/\tau^4 = 5000/(1-\Pi)^4$, 
so that Lemma~\ref{lem:varphi} implies $1 \le 5000 \varphi(-\eps)/(1-\Pi)^4$. 
\end{proof}

\subsection{Proofs of Theorem~\ref{thm:LT} and Lemma~\ref{lem:LT4}}\label{sec:lbprof:comb}
The remaining proofs of Theorem~\ref{thm:LT} and Lemma~\ref{lem:LT4} 
are straightforward. 
\begin{proof}[Proof of Lemma~\ref{lem:LT4}]
Note that, by assumption, $5\sqrt{1-\eps} \le 5 e^{-1/2} \le 4$. So, 
using Lemma~\ref{lem:varphi2}, we infer 
\[
1/(1-\Pi) \le (1+5\sqrt{1-\eps})\bigl(1+\Pi/(1-\Pi)\bigr)\varphi(-\eps) \le (1+\zeta)\varphi(-\eps) ,
\]
with $\zeta = 10\max\{\sqrt{1-\eps},\Pi/(1-\Pi)\}$. Now an application 
of~\eqref{eq:H}, analogous to \eqref{eq:harris:LB}, completes the proof. 
\end{proof}
\begin{proof}[Proof of Theorem~\ref{thm:LT}]
Note that, using the assumption, 
\[
\eta = \max\{4\Pi^{1/4},\indic{\eps<1}4\delta^{1/4},\indic{\eps<1}e^{-1}(\eps^2\mu)^{-1/2}\} 
\] 
satisfies $\eta \in [0,e^{-1}]$. 
If $1-\eta \le \eps \le 1$, then $\eps \ge 1-e^{-1}$ and 
$1-\eps \le \eta$, so that Lemma~\ref{lem:LT4} implies \eqref{eq:LT}. 
If $0 \le \eps < 1-\eta$, then $e(1-\eps)\eps^2\mu \ge e\eta\eps^2\mu\ge (\eps^2\mu)^{1/2} \ge 1$ 
and $\eps \le 1-4\max\{\Pi^{1/4},\delta^{1/4}\}$, so that 
Theorem~\ref{thm:LT3} establishes \eqref{eq:LT}. 
\end{proof}

\section{Bootstrapping lower bounds for the lower tail}\label{sec:bstrp} 
As discussed, Theorem~\ref{thm:LT} and~\ref{thm:LT2} only give 
reasonable lower bounds for the lower tail if $\delta = O(1)$, i.e., 
as long as the dependencies are `weak'. In this section we present a 
bootstrapping strategy, which often allows us to deal with the 
remaining case, where $\delta = \Omega(1)$ holds.

In order to establish a competent lower bound on the lower tail, we 
usually need to (approximately) identify the most likely way to 
obtain $X \le (1-\eps)\E X$. At first glance it seems that this would 
require fairly detailed information about the random variable $X$, 
where $\mu= \E X$. However, in the general setting of this paper, we 
discovered that, perhaps surprisingly, we can \emph{systematically} 
guess suitable (nearly) `extremal' events by only inspecting the form 
of the variance $\Var X \le \Lambda=\Lambda(X)$. Indeed, assume that 
there is a random variable $Y$, of the same type as \eqref{def:X}, 
satisfying 
\begin{equation}\label{meth:V}
\Lambda = \Theta(\mu^2/\E Y) \quad \text{and} \quad \delta(Y)=O(1) .
\end{equation}
For example, if $X_{H}$ counts the number of copies of a given graph 
$H$ in $G_{n,p}$, then \eqref{meth:V} holds for $X=X_{H}$ with $Y=X_G$, 
where $G \subseteq H$ is a suitable subgraph (see \cite{JLR_ineq,JLR} 
or Lemma~\ref{lem:lambda:HG}). Defining $\cE$ as the event that 
$Y \le (1-\eps)\E Y$ holds, our starting point is the basic inequality 
\begin{equation}\label{meth:Pr}
\Pr(X \le (1-\eps)\E X) \ge \Pr(X \le (1-\eps)\E X \mid \cE) \Pr(\cE) . 
\end{equation}
Assuming that Theorem~\ref{thm:LT} or~\ref{thm:LT2} applies to $Y$, 
using \eqref{meth:V} there are constants $c_1,c_2 > 0$ such that 
\begin{equation}\label{meth:PrU}
\Pr(\cE) \ge e^{-c_1 \varphi(-\eps) \E Y} \ge e^{-c_2 \varphi(-\eps) \mu^2/\Lambda} .
\end{equation}
Hence it remains to estimate $\Pr(X \le (1-\eps)\E X \mid \cE)$ from 
below. It turns out that if $X$ and $Y$ are suitably related (as in 
the subgraphs example), then under fairly mild conditions we can prove 
that $\E(X \mid \cE)$ is quite a bit smaller than $(1-\eps)\E X$. 
In other words, by conditioning on $\cE$ we intuitively `convert' the 
\emph{rare} event $X \le (1-\eps)\E X$ into a \emph{typical} one 
(this subtle conditioning idea is at the heart of our approach). 
With this in mind it seems plausible that we have, say, 
\begin{equation}\label{meth:Cond}
\Pr(X \le (1-\eps)\E X\mid\cE) = \Omega(1) ,
\end{equation}
although $\ge e^{-c_3 \varphi(-\eps) \mu^2/\Lambda}$ suffices for 
our purposes. Note that for the special case $\eps=1$ this inequality 
is immediate in the subgraphs example (where $X_G=0$ implies $X_H=0$). 
Finally, by combining \eqref{meth:Pr}--\eqref{meth:Cond} we obtain 
\begin{equation}\label{meth:Pr:LB}
\Pr(X \le (1-\eps)\E X) = \Omega(e^{-c_2 \varphi(-\eps) \mu^2/\Lambda)}) ,
\end{equation}
which qualitatively matches the upper bound of \eqref{eq:J}, as desired.

To implement this proof strategy, we need to be able to verify that 
\eqref{meth:Cond} holds (or a related inequality). Here the main 
technical challenge is that, after conditioning on $\cE$, the 
$i \in \Gamma$ are no longer added independently to $\Gamma_{\vp}$. 
In Sections~\ref{sec:bstrp:size}--\ref{sec:bstrp:vxsym} we present 
three approaches that, in symmetric situations, allow us to 
\emph{routinely} overcome this difficulty (each of them hinges on 
an event that is similar to $\cE$). 
Since we are interested in large deviations (with exponentially 
small probabilities), here $(\eps \mu)^2 = \Omega(\Lambda)$ is a 
natural condition in view of \eqref{eq:J}, \eqref{meth:Pr:LB} and 
the fact $\varphi(-\eps) = \Theta(\eps^2)$.

\subsection{Binomial random subset}\label{sec:bstrp:size}
The first approach is motivated by the following simple observation: 
if $|\Gamma_{\vp}|=0$, then \emph{deterministically} $X=0$. Indeed, 
this yields 
\begin{equation*}\label{eq:rsize:0}
\Pr(X \le (1-\eps) \E X) \ge \Pr(X = 0 ) \ge \Pr(|\Gamma_{\vp}| =0) , 
\end{equation*}
which for $\eps=\Theta(1)$ may give a fair lower bound. 
The next theorem, for the case of equal $p_i$, is based on the 
following heuristic extension of this observation: if $|\Gamma_{\vp}|$ 
is `too small', then we expect that $X$ is \emph{typically} also 
`too small'. As we shall see, the crux is that conditioning on 
$|\Gamma_{\vp}| \le (1-\eps) \E |\Gamma_{\vp}|$ decreases the 
expected value of $X$, which intuitively increases the probability 
that $X \le (1-\eps) \E X$ occurs. Note that $\E(X \mid |\Gamma_{\vp}|=0)=0$ 
confirms this phenomenon in the special case $\eps=1$. 
\begin{theorem}\label{thm:rsize}
Let $X=\sum_{\alpha \in \cX} I_{\alpha}$, $\mu=\E X$ and $\Lambda$ 
be defined as in Section~\ref{sec:intro}. 
Suppose that $\vp=(p, \ldots, p) \in [0,1]^N$ and $\min_{\alpha \in \cX}|Q(\alpha)| \ge 2$. 
For all $\eps \in (0,1]$ satisfying $(\eps \mu)^2 \ge \indic{\eps < 1}\Lambda$, with $c=1/2 + \indic{\eps =1}1/2$,
\begin{equation}\label{eq:rsize}
\Pr(X \le (1-\eps) \E X) \ge c \Pr(|\Gamma_{\vp}| \le (1-\eps) \E |\Gamma_{\vp}|) .
\end{equation}
\end{theorem}
In the proof of Theorem~\ref{thm:rsize} we use the following 
one-sided version of Chebyshev's inequality (see, e.g., 
Theorem~A.17 in~\cite{DGL1996}). 
\begin{claim}\label{cl:Ch}
If $\Var Z \le v$, then $\Pr(Z \ge \E Z + t) \le v/(v+t^2)$ for all $t > 0$. 
\end{claim}
\begin{proof}[Proof of Theorem~\ref{thm:rsize}]
Given $0 \le j \le N$, we write $\Pr(\cdot \mid |\Gamma_{\vp}|=j) = \Pr_j(\cdot)$ for brevity. 
Note that for $m=(1-\eps)Np=(1-\eps)\E |\Gamma_{\vp}|$ we have 
\begin{equation}\label{eq:thm:rsize:Pr}
\begin{split}
\Pr(X \le (1-\eps)\mu) &\ge \sum_{0 \le j \le m}\Pr_j(X \le (1-\eps)\mu) \Pr(|\Gamma_{\vp}|=j)\\
& \ge \Pr(|\Gamma_{\vp}|\le m) \min_{0 \le j \le m}\Pr_j(X \le (1-\eps)\mu) .
\end{split}
\end{equation}
Since $\Pr_0(X \le (1-\eps)\mu) \ge \Pr_0(X =0)=1$, we henceforth may 
assume $m \ge 1$. Consequently $\eps < 1$ and $p > 0$ hold, so that 
$\mu \ge \min_{\alpha \in \cX} \E I_{\alpha} \ge p^N > 0$. 

In the following we estimate the conditional expected value and 
variance of~$X$. Given $0 \le j \le m$, we write $\E(\cdot \mid |\Gamma_{\vp}|=j) = \E_j(\cdot)$ 
and $\Var(\cdot \mid |\Gamma_{\vp}|=j) = \Var_j(\cdot)$ for brevity. 
Let $\Gamma_{j} \subseteq \Gamma$ with $|\Gamma_{j}|=j$ be chosen 
uniformly at random. Since $\vp=(p, \ldots, p)$, it follows that 
$\Gamma_{\vp}$ conditioned on $|\Gamma_{\vp}|=j$ has the same 
distribution as $\Gamma_{j}$. As $|Q(\alpha)| \ge 2$ and 
$j \le m \le N$, using $I_{\alpha} = \indic{Q(\alpha) \subseteq \Gamma_{\vp}}$ 
we infer
\begin{equation}\label{eq:thm:rsize:E}
\begin{split}
\E_j(I_\alpha) &= \indic{|Q(\alpha)| \le j}\frac{\binom{N-|Q(\alpha)|}{j-|Q(\alpha)|}}{\binom{N}{j}} = \indic{|Q(\alpha)| \le j}\prod_{0 \le i < |\alpha|}\frac{j-i}{N-i} \\
& \le \Bigl(\frac{j}{N}\Bigr)^{|Q(\alpha)|} \le (1-\eps)^{|Q(\alpha)|} p^{|Q(\alpha)|} \le (1-\eps)^2 \E I_{\alpha} .
\end{split}
\end{equation}
Since $I_\alpha I_\beta=\indic{Q(\alpha) \cup Q(\beta) \subseteq \Gamma_{\vp}}$, we analogously 
obtain $\E_j(I_\alpha I_\beta) \le (1-\eps)^2 \E(I_\alpha I_\beta)$. 
Furthermore, if $Q(\alpha) \cap Q(\beta) = \emptyset$ and 
$|Q(\alpha)|+|Q(\beta)| \le j$, then a similar calculation shows that 
\begin{equation*}
\E_j(I_\alpha \mid I_{\beta}=1) = \frac{\binom{N-|Q(\beta)|-|Q(\alpha)|}{j-|Q(\beta)|-|Q(\alpha)|}}{\binom{N-|Q(\beta)|}{j-|Q(\beta)|}} = \prod_{0 \le i < |Q(\alpha)|} \frac{j-|Q(\beta)|-i}{N-|Q(\beta)|-i} \le \E_j(I_\alpha) .
\end{equation*}
If $|Q(\alpha) \cup Q(\beta)| > j$ then, trivially, $\E_j(I_\alpha I_{\beta}) =0$. 
It follows that $Q(\alpha) \cap Q(\beta) = \emptyset$ implies 
$\E_j(I_\alpha I_{\beta})-\E_j(I_\alpha) \E_j(I_{\beta}) \le 0$. 
Combining our findings, we deduce that 
\begin{equation}\label{eq:thm:rsize:EVar}
\max_{0 \le j \le m}\E_j(X) \le (1-\eps)^2 \mu \quad \text{ and } \quad \max_{0 \le j \le m}\Var_j(X) \le (1-\eps)^2 \Lambda .
\end{equation}

Finally, using \eqref{eq:thm:rsize:EVar} and the one-sided Chebyshev's 
inequality (Claim~\ref{cl:Ch}) we infer that for every $0 \le j \le m$ 
we have 
\begin{equation*}\label{eq:thm:rsize:PrCh}
\Pr_j(X > (1-\eps)\mu) \le \Pr_j(X \ge \E_j(X) + (1-\eps) \eps \mu) \le \Lambda/(\Lambda+(\eps\mu)^2) ,
\end{equation*}
which together with $(\eps\mu)^2 \ge \Lambda$ and \eqref{eq:thm:rsize:Pr} 
establishes \eqref{eq:rsize}. 
\end{proof}
The proof shows that \eqref{eq:rsize} holds with $c$ replaced by 
$1-\indic{\eps < 1, \mu>0}\Lambda/(\Lambda+(\eps\mu)^2)$, and that the 
left hand side of \eqref{eq:rsize} can be strengthened to 
$\Pr(X < (1-\eps) \E X)$ whenever $\eps \in (0,1)$ and $\mu > 0$ (we 
henceforth omit analogous remarks).

In applications where constant factors in the exponent are important, 
the following variant of Theorem~\ref{thm:rsize} usually gives better 
results when $\eps \to 0$ and $L=(\eps \mu)^2/\Lambda \to \infty$ 
(by setting $\tau = 6 \max\{\eps, L^{-1/2}\}$; see 
Lemma~\ref{lem:varphi3} with $A=(1+\tau)/k$). 
\begin{theorem}\label{thm:rsize2}
Let $X=\sum_{\alpha \in \cX} I_{\alpha}$, $\mu=\E X$ and $\Lambda$ be 
defined as in Section~\ref{sec:intro}. 
Suppose that $\vp=(p, \ldots, p) \in [0,1]^N$ and $\min_{\alpha \in \cX}|Q(\alpha)| \ge k \ge 1$. 
For all $\eps,\tau \in (0,1]$ satisfying $\tau \ge \indic{k > 1}6\eps$ 
and $(\eps \mu)^2 \ge 4\tau^{-2}\Lambda$, with $c=1/2$, 
\begin{equation}\label{eq:rsize2}
\Pr(X \le (1-\eps) \E X) \ge c \Pr(|\Gamma_{\vp}| \le (1-(1+\tau)\eps/k) \E |\Gamma_{\vp}|) .
\end{equation}
\end{theorem}
\begin{proof}
Let $\lambda = (1+\tau)\eps/k$ and $m=(1-\lambda)\E |\Gamma_{\vp}|$. 
As \eqref{eq:rsize2} is trivial otherwise, we henceforth assume 
$\Pr(|\Gamma_{\vp}| \le m) > 0$, which implies $m \ge 0$. 
Now, \eqref{eq:thm:rsize:Pr} carries over mutatis mutandis, and, 
with similar reasoning as in the proof of Theorem~\ref{thm:rsize}, 
we may henceforth assume $\min\{m,p,\mu\}>0$. 
Furthermore, as $\min_{\alpha \in \cX}|Q(\alpha)| \ge k$, the 
calculations leading to \eqref{eq:thm:rsize:EVar} imply 
\begin{equation}\label{eq:thm:rsize2:EVar}
\max_{0 \le j \le m}\E_j(X) \le (1-\lambda)^k \mu \quad \text{ and } \quad \max_{0 \le j \le m}\Var_j(X) \le \Lambda .
\end{equation}
If $k=1$, then $(1-\eps) - (1-\lambda)^k = \lambda-\eps=\tau \eps$, 
and we now establish a similar bound for $k > 1$. 
Note that $\lambda k = (1+\tau)\eps \le 2 \eps \le \tau/3 < 1$ and 
\[
(1-\lambda)^k \le e^{-\lambda k} \le 1 - \lambda k + \sum_{j \ge 2} \frac{(\lambda k)^j}{j!} \le 1-\lambda k + \frac{(\lambda k)^2}{2(1-\lambda k)} .
\]
Recalling $\lambda k = (1+\tau)\eps$, $\eps \le \tau/6$ and 
$\tau \le 1$, a short calculation shows that 
\[
(1-\eps) - (1-\lambda)^k \ge \tau \eps\left(1 - \frac{(1+\tau)^2\eps}{2\tau(1-(1+\tau)\eps)}\right) \ge \tau \eps /2 .
\]
Consequently, using \eqref{eq:thm:rsize2:EVar} and the one-sided 
Chebyshev's inequality (Claim~\ref{cl:Ch}), we infer that for 
every $0 \le j \le m$ we have 
\begin{equation*}\label{eq:thm:rsize2:PrCh}
\Pr_j(X > (1-\eps)\mu) \le \Pr_j(X \ge \E_j(X) + \tau \eps \mu /2) \le \Lambda/(\Lambda+ \tau^2(\eps\mu)^2/4) ,
\end{equation*}
which together with $(\eps\mu)^2 \ge 4 \tau^{-2}\Lambda$ and 
\eqref{eq:thm:rsize:Pr} establishes \eqref{eq:rsize2}. 
\end{proof}

\subsection{Symmetric decomposition}\label{sec:bstrp:sdcmp}
In general, the conditional expected value of $X$ is difficult to 
compute (as we do not have explicit formulas as in \eqref{eq:thm:rsize:E}). 
Our second approach shows that we can overcome this obstacle using a 
\emph{symmetric decomposition} of $X$. As an illustration, we again 
consider the number of copies of $H$ in $G_{n,p}$. Clearly, for every 
$G \subseteq H$ we have $\Pr(X_H=0) \ge \Pr(X_G=0)$. 
The basic idea is now that, by counting the number of $H$-copies 
extending each copy of $G$, we ought to be able to argue as follows: 
if $X_G$ is `too small', then the (conditional) expected value of $X_H$ 
is also `too small'. To avoid clutter, we henceforth use the abbreviation 
\begin{equation}\label{conv:Iab}
I_{\alpha \setminus \beta} = \indic{Q(\alpha) \setminus Q(\beta) \subseteq \Gamma_{\vp}} .
\end{equation}
Let $\cH=\cH_n$ contain all subgraphs isomorphic to $H$ in $K_n$, and 
define $Q(\alpha)=E(\alpha)$ for all $\alpha \in \cH$ (here 
$Q(\alpha) \neq \alpha$ is crucial to allow for isolated vertices in $H$). 
The key observation is that, by symmetry, there is a constant $w>0$ such 
that we may write 
\[
X_H = w\sum_{\beta \in \cG} I_{\beta} \sum_{\alpha \in \cH: \beta \subseteq \alpha} I_{\alpha \setminus \beta}, 
\]
 where $\E \bigl(\sum_{\alpha \in \cH: \beta \subseteq \alpha} I_{\alpha \setminus \beta}\bigr)$ 
is \emph{independent} of the choice of $\beta \in \cG$. 
The point is that, since $\E(I_{\beta}I_{\alpha \setminus \beta})=\E I_{\beta} \E I_{\alpha \setminus \beta}$ 
and $X_{G}=\sum_{\beta \in \cG}I_{\beta}$, this allows us to factorize 
$\E X_H$ in terms of $\E X_{G}$. Indeed, for any $\tilde{\beta} \in \cG$ we have 
\[
\E X_H = w \E \bigl(\sum_{\alpha \in \cH: \tilde{\beta} \subseteq \alpha} I_{\alpha \setminus \tilde{\beta}}\bigr) \sum_{\beta \in \cG}\E I_{\beta} = w \E X_{G} \E \bigl(\sum_{\alpha \in \cH: \tilde{\beta} \subseteq \alpha} I_{\alpha \setminus \tilde{\beta}}\bigr) .
\]
Intuitively, our approach exploits that correlation inequalities 
can be used to obtain a similar factorization of the \emph{conditional} 
expected value of $X_H$.

With the subgraphs example in mind, the following theorem should be 
interpreted under the premise that the lower bound is exponentially 
small in $\Theta((\eps \mu)^2/\Lambda)$. In other words, the 
multiplicative $\gamma\eps$ error-term ought to be negligible as 
long as, say, $\gamma\eps \ge e^{-(\eps \mu)^2/\Lambda}$ holds. 
The crux is that this inequality is equivalent to 
$(\eps \mu)^2/\Lambda \ge \log\bigl(1/(\gamma\eps)\bigr)$, which 
matches our usual condition up to the logarithmic factor. 
On first reading it might be useful to consider the important 
special case exemplified above, where $w_{\alpha,\beta}=w>0$, 
$\cX(\beta)=\{\alpha \in \cX: Q(\beta) \subseteq Q(\alpha)\}$ and 
$\kappa=0$. 
\begin{theorem}\label{thm:rcor}
Let $Y=\sum_{\beta \in \cY} I_{\beta}$, where $\bigl(Q(\beta)\bigr)_{\beta \in \cY}$ 
is a family of subsets of $\Gamma$. 
Suppose that there are $w_{\alpha,\beta} \in [0,\infty)$ and families 
$\bigl(Q(\alpha)\bigr)_{\alpha \in \cX(\beta)}$ of subsets of $\Gamma$ 
such that $X = \sum_{\beta \in \cY} I_{\beta} X_{\beta}$, where 
$X_{\beta} = \sum_{\alpha \in \cX(\beta)} w_{\alpha,\beta} I_{\alpha \setminus \beta}$ 
satisfies $\max_{\beta \in \cY} \E X_{\beta} \le (1+\kappa) \min_{\beta \in \cY} \E X_{\beta}$ for $\kappa \in [0,\infty)$. 
For all $\eps \in [0,1]$ and $\gamma \in [0,\infty)$ satisfying 
$\gamma \eps \ge 2\kappa$ and $\indic{\E Y=0}\gamma \eps \le 2$, with $c=1/2$, 
\begin{equation}\label{eq:rcor}
\Pr(X \le (1-\eps) \E X) \ge c \gamma\eps \Pr(Y \le (1-(1+\gamma)\eps) \E Y) .
\end{equation}
\end{theorem}
If $\eps \nearrow 1$ or $\eps=1$ holds, then, by applying Lemma~\ref{lem:LT4} 
to $Y$, we often can improve \eqref{eq:rcor} via 
\begin{equation}\label{eq:rcor:0}
\Pr(X \le (1-\eps) \E X) \ge \Pr(X =0) \ge \Pr(Y=0) .
\end{equation}
The proof of Theorem~\ref{thm:rcor} hinges on the following simple 
consequence of Harris' inequality~\cite{Harris1960}, 
which was observed by Bollob{\'a}s and Riordan 
(see Lemma~6 in~\cite{BR1998}).
\begin{claim}\label{cl:cor}
For the probability space induced by $\Gamma_{\vp}$, suppose that 
$\cD$ is a decreasing event with $\Pr(\cD)>0$, and that $\cI_1$ and 
$\cI_2$ are increasing events with 
$\Pr(\cI_1 \cap \cI_2) = \Pr(\cI_1) \Pr(\cI_2)$. Then 
\begin{equation}\label{eq:cl:cor}
\Pr(\cI_1 \cap \cI_2 \mid \cD) \le \Pr(\cI_1) \Pr(\cI_2 \mid \cD) .
\end{equation}
\end{claim}
\begin{proof}[Proof of Theorem~\ref{thm:rcor}]
Let $y=(1-(1+\gamma)\eps) \E Y$ and $\mu=\E X$. As \eqref{eq:rcor} is 
trivial otherwise, we henceforth assume $\gamma \eps >0$ and 
$\Pr(Y \le y)>0$, which since $Y \ge 0$ implies $y \ge 0$. 
If $\E Y = 0$, then $\Pr(Y=0)=\Pr(Y \le y)$, and, since we then assume 
$1 \ge \gamma\eps/2$, 
\eqref{eq:rcor:0} establishes \eqref{eq:rcor}. 
Henceforth we thus assume $\E Y > 0$, so that $y \ge 0$ implies 
$1 \ge (1+\gamma)\eps > \max\{\eps,\gamma \eps\}$. Note that 
\begin{equation}\label{eq:thm:rcor:Pr}
\Pr(X \le (1-\eps) \mu) \ge \Pr(Y \le y) \Pr(X \le (1-\eps)\mu \mid Y \le y) .
\end{equation}
Since $\E(I_{\beta}I_{\alpha \setminus \beta}) = \E I_{\beta} \E I_{\alpha \setminus \beta}$, 
using the definitions of $X$, $X_{\beta}$ and $Y$ we deduce 
\begin{equation}\label{eq:thm:rcor:E}
\mu = \E X = \sum_{\beta \in \cY} \E I_{\beta} \E X_{\beta} \ge \E Y \min_{\beta \in \cY} \E X_{\beta} \ge (1+\kappa)^{-1} \E Y \max_{\beta \in \cY} \E X_{\beta} .
\end{equation}
We write $\cI_{\alpha}$ and $\cI_{\alpha \setminus \beta}$ for the 
increasing events that $I_{\alpha}=1$ and $I_{\alpha \setminus \beta}=1$, 
respectively. Hence $\Pr(\cI_{\alpha\setminus\beta}\cI_{\beta})=\Pr(\cI_{\alpha\setminus\beta})\Pr(\cI_{\beta})$. 
Clearly, $Y \le y$ is a decreasing event. Using Claim~\ref{cl:cor} 
together with \eqref{eq:thm:rcor:E} and $(1-(1+\gamma)\eps)(1+\kappa) \le 1-(1+\gamma/2)\eps$, 
it follows that
\begin{equation}\label{eq:thm:rcor:EC}
\begin{split}
\E(X \mid Y \le y) & = \sum_{\beta \in \cY} \sum_{\alpha \in \cX(\beta)} w_{\alpha,\beta} \Pr(\cI_{\alpha\setminus\beta}\cI_{\beta} \mid Y \le y) \le \sum_{\beta \in \cY} \Pr(\cI_{\beta} \mid Y \le y) \sum_{\alpha \in \cX(\beta)} w_{\alpha,\beta} \Pr(\cI_{\alpha\setminus\beta}) \\
& \le \E(Y \mid Y \le y) \max_{\beta \in \cY} \E X_{\beta} 
 \le  y \max_{\beta \in \cY} \E X_{\beta} 
\le (1-(1+\gamma/2)\eps) \mu.
\end{split}
\end{equation}
Let $\lambda = 1+\gamma/2$. If $\mu > 0$, then, using Markov's 
inequality, we infer from \eqref{eq:thm:rcor:EC} 
\begin{equation}\label{eq:thm:rcor:PrM}
\Pr(X > (1-\eps)\mu \mid Y \le y) \le \frac{1-\lambda\eps}{1-\eps} = 1-\frac{(\lambda-1)\eps}{1-\eps} \le 1-\gamma\eps/2 ,
\end{equation}
which together with \eqref{eq:thm:rcor:Pr} establishes 
\eqref{eq:rcor}. Finally, if $\mu=0$, then 
 $\Pr(X>0)=0$ and \eqref{eq:rcor} follows trivially from the fact
$1>\gamma\eps$ established above.
\end{proof}
It would be desirable to use Chebyshev's inequality in \eqref{eq:thm:rcor:PrM}, 
since this presumably would improve the seemingly suboptimal 
$\gamma\eps$ term. Here one technical obstacle is that Claim~\ref{cl:cor} 
can, in general, \emph{not} be strengthened to 
\begin{equation}\label{eq:cl:cor:wrong}
\Pr(\cI_1 \cap \cI_2 \mid \cD) \le \Pr(\cI_1 \mid \cD) \Pr(\cI_2 \mid \cD) . 
\end{equation}
Indeed, a short calculation shows that, for $\Gamma=[n]=\{1,\ldots,n\}$ 
and $\vp=(p, \ldots,p)$ with $n \ge 3$ and $p \in (0,1)$, the events 
$\cI_i=\{i \in \Gamma_{\vp}\}$ and $\cD=\{|\Gamma_{\vp}| \le 1 \text{ or } \Gamma_{\vp}=\{1,2\}\}$ 
provide a counterexample (where, moreover, equality holds in 
\eqref{eq:cl:cor}). It would be interesting to know whether there is 
perhaps some approximate version of \eqref{eq:cl:cor:wrong} that 
suffices for our purposes.

The existence of a symmetric decomposition may not always be obvious. 
We hope that the following two examples from additive combinatorics 
serve as inspiration for future applications of Theorem~\ref{thm:rcor} 
(or its method of proof). In both we consider $\vp=(p, \ldots, p)$ and 
$Q(\alpha)=\alpha$, and the basic idea is to `symmetrize' $X$ using 
non-uniform `weights' $w_{\alpha,\beta}$ (and $\kappa\neq 0$). 
In the first example, we let $\cX$ contain all \emph{arithmetic 
progressions} of length $k \ge 2$ in $\Gamma=[n]$, i.e., each 
$\alpha \in \cX$ equals $\{b,b+d, \ldots, b+(k-1)d\} \subseteq [n]$ 
for some $b=b_{\alpha}$ and $d=d_{\alpha}$ with $b_{\alpha},d_{\alpha} \ge 1$. 
For every $\beta \in \cY=[n]$ we define $\cX(\beta)$ as the set of 
$\alpha \in \cX$ where $\beta=b_{\alpha}$ or $\beta=b_{\alpha}+(k-1)d_{\alpha}$, 
and set $w_{\alpha,\beta}=1/2$. Since each $\alpha \in \cX$ 
contributes to exactly two $X_{\beta}$, we have 
$X = \sum_{\beta \in \cY} I_{\beta}X_\beta$. Furthermore, careful 
counting yields 
\[
\E X_{\beta}=\frac{1}{2}\Bigl(\floorBL{\frac{n-\beta}{k-1}}+\floorBL{\frac{\beta-1}{k-1}}\Bigr)p^{k-1}=\Bigl(\frac{n}{2(k-1)}+O(1)\Bigr)p^{k-1}, 
\]
so $\kappa = O(1/n)$ suffices. 
In the second example, we let $\cX$ contain all \emph{Schur triples} 
in $\Gamma=[n]$, i.e., each $\alpha \in \cX$ equals $\{x,y, x+y\} \subseteq [n]$ 
for some $x=x_{\alpha}$ and $y=y_{\alpha}$ with $1 \le x_{\alpha} < y_{\alpha}$. 
For every $\beta \in \cY=[n]$ we define $\cX(\beta)$ as the set of 
all $\alpha \in \cX$ with $\beta \in \alpha$. We set $w_{\alpha,\beta}=1/2$ 
if $\beta = x_{\alpha}+y_{\alpha}$, and $w_{\alpha,\beta}=1/4$ otherwise. 
By counting triples, it is not hard to see that 
$X = \sum_{\beta \in \cY} I_{\beta}X_\beta$ and 
\[
\E X_{\beta}=\Bigl(\frac{1}{2}\floorBL{\frac{\beta-1}{2}} + \frac{\max\{n-2\beta,0\} + \min\{n-\beta,\beta-1\}}{4}\Bigr)p^2 = \Bigl(\frac{n}{4}+O(1)\Bigr)p^2, 
\] 
so $\kappa = O(1/n)$ suffices. 
Finally, in both examples routine calculations (analogous to Example~3.2 
in~\cite{JLR}) give $\mu^2/\Lambda = \Theta(\min\{\mu,np\})$. 
Since $\kappa = O(1/n)$ and $\mu^2/\Lambda=O(np)$, the natural condition 
$(\eps\mu)^2 =\Omega(\Lambda)$ thus implies $\kappa/\eps = O(1/n \cdot \sqrt{\mu^2/\Lambda}) = O(\sqrt{p/n})=o(1)$. 
In other words, the assumption $\gamma \eps \ge 2\kappa$ in 
Theorem~\ref{thm:rcor} is very mild, i.e., allows for $\gamma=o(1)$.

\subsection{Vertex symmetry}\label{sec:bstrp:vxsym}
In many applications the set $\Gamma$ has additional structure, and 
here our main focus is on the case where $\Gamma$ contains the edges 
of some hypergraph. Intuitively, `seeing' the underlying vertices 
introduces quite a bit of extra symmetry, and our third approach 
exploits this to step aside the conditioning issue we faced in the 
previous subsection. 
As an illustration, we consider, as before, the number of copies of 
$H$ in $G_{n,p}$. The basic idea is to partition the vertex set into 
$\cU$ and $[n] \setminus \cU$ with $|\cU| \approx n/2$, and then, for 
suitable $G \subseteq H$, to focus on the number of copies of $G$ 
completely contained in $\cU$, which we denote by $Y_G$. Note that 
$\E Y_G = \Theta(\E X_G)$. Perhaps rashly, we would like to argue 
that $Y_G \le (1-\eps) \E Y_G$ typically entails $X_H \le (1-\eps) \E X_H$. 
However, this is overly ambitious: since $Y_G$ is somewhat `local', 
we loose a bit when going to the `global' random variable $X_H$, 
and thus we need a slightly larger deviation of $Y_G$. 
Instead of counting all copies of $H$, a technical reduction allows 
us to focus on the number of \emph{pairs} $(H',G')$ of copies of $H$ 
and $G$ with $G' \subseteq H'$, $V(G') \subseteq \cU$ and 
$V(H') \setminus V(G') \subseteq [n] \setminus \cU$. 
Now, to make variance calculations feasible (i.e., to overcome the 
obstacle that \eqref{eq:cl:cor:wrong} may fail), we do not condition 
on $Y_G$, but rather on \emph{all edges} with both endvertices in $U$ 
(satisfying additional typical properties). For technical reasons, 
here our argument requires that all edges in the relevant graphs 
$H' \setminus G'$ have at least one endvertex outside of $\cU$, 
which, e.g., holds if all copies of $G$ in $H$ are \emph{induced} 
subgraphs. Luckily, it is not hard to check (see Lemma~\ref{lem:lambda:HG}) 
that the former condition always holds for some $G \subseteq H$ that 
determines the exponent, i.e., satisfies $\Lambda(X_H) = \Theta((\E X_H)^2/\E X_G)$. 

In the statement of the next theorem we restrict ourselves to subgraph 
counts in random hypergraphs. The approach works in a more general 
setting, but we resist the temptation of stating a very technical 
theorem (that would be difficult to apply). Instead, we tried to write 
the proof in a way that hopefully makes the basic setup and symmetry 
assumptions fairly transparent. 
In Theorem~\ref{thm:vxsym} the difference between $Y_G$ and $X_G$ is 
usually irrelevant in applications where constant factors in the 
exponent are immaterial: the point is that $G_{n,p}^{(k)}[\cU]$ has 
the same distribution as $G_{n',p}^{(k)}$ with $n'=|\cU|\approx n/2$. 
In comparison with Theorem~\ref{thm:rcor}, the key feature of 
Theorem~\ref{thm:vxsym} is that the natural condition 
$(\eps\E X_H)^2 = \Omega(\Lambda(X_H))$ suffices. 
\begin{theorem}\label{thm:vxsym}
Let $G \subseteq H$ be $k$-graphs with $e_G \ge 1$, where every copy 
of $G$ in $H$ is induced. Let $X_H$ be the number of copies 
of $H$ in $G^{(k)}_{n,p}$, and let $Y_G$ be the number of copies of 
$G$ in $G^{(k)}_{n,p}[\cU]$, where $\cU \subseteq [n]$ satisfies 
$\bigl||\cU|-n/2\bigr| \le \ell$. For all $n \ge n_0=n_0(H,\ell)$, 
$p \in [0,1]$ and $\eps \in (0,1]$ satisfying
$(\eps \E X_H)^2 \ge \Lambda(X_H)$, with $\lambda = 2^{v_H+3}$ and 
$c=2^{-(4^{v_G^2}+2)}$, 
\begin{equation}\label{eq:vxsym}
\Pr(X_H \le (1-\eps) \E X_H) \ge c \Pr(Y_G \le (1-\lambda\eps) \E Y_G) .
\end{equation}
\end{theorem}
\begin{proof}
Let $\mu=\E X_H$, $\Lambda = \Lambda(X_H)$, $\Gamma=E(K^{(k)}_n)$ and 
$\vp=(p, \ldots, p)$, so that $\Gamma_{\vp}=E(G^{(k)}_{n,p})$. Let $\cH$ 
and $\cG$ contain all subgraphs isomorphic to $H$ and $G$ in 
$K^{(k)}_n$, respectively. Define $Q(\sigma)=E(\sigma)$ for $\sigma \in \cH \cup \cG$. 
For brevity we henceforth use $I_{(\alpha_1 \cup \alpha_2) \setminus (\beta_1 \cup \beta_2)} = \indic{[Q(\alpha_1) \cup Q(\alpha_2)] \setminus [Q(\beta_1) \cup Q(\beta_2)] \subseteq \Gamma_{\vp}}$ 
and $I_{\sigma_1 \cup \sigma_2} = \indic{Q(\sigma_1) \cup Q(\sigma_2) \subseteq
  \Gamma_{\vp}}$
 analogous to \eqref{conv:Iab}.
Set $Z=\sum_{(\alpha,\beta) \in \cH \times \cG} \indic{\beta \subseteq \alpha}I_{\alpha}$. 
By symmetry, we have $\sum_{\beta \in \cG} \indic{\beta \subseteq \alpha}=\tau=\tau(H,G) \ge 1$ 
for all $\alpha \in \cH$. Hence $Z=\tau X$, $\E Z = \tau \E X_H$, 
$\Var Z = \tau^2 \Var X_H$ and 
\begin{equation}\label{eq:vxsym:Pr:1}
\Pr(X_H \le (1-\eps) \E X_H) = \Pr(Z \le (1-\eps) \E Z). 
\end{equation}
With foresight, we set $Z_S = \sum_{(\alpha,\beta) \in \cH \times \cG} \indic{\alpha \in \cH(S,\beta) \text{ and } \beta \in \cG(S)} I_{\alpha}$ 
for all $S \subseteq [n]$, where 
\begin{equation*}
\begin{split}
\cH(S,\beta)&=\{\alpha \in \cH: \beta \subseteq \alpha \text{ and } V(\alpha) \setminus V(\beta) \subseteq [n] \setminus S\}, \\
\cG(S) & =\{\beta \in \cG: V(\beta) \subseteq S\} .
\end{split}
\end{equation*}
Define $R_{\cU}=Z-Z_{\cU}$, $z=(1-\eps\lambda/2) \E Z_{\cU}$ and 
$r=(1-\eps) \E Z-z$. Using $Z=R_{\cU}+Z_{\cU}$ and Harris' inequality, 
it follows that 
\begin{equation}\label{eq:vxsym:Pr:2}
\Pr(Z \le (1-\eps) \E Z) \ge \Pr(R_{\cU} \le r \text{ and } Z_{\cU} \le z) \ge \Pr(R_{\cU} \le r) \Pr(Z_{\cU} \le z) . 
\end{equation}
The remainder of the proof is devoted to the following two 
inequalities, which together with \eqref{eq:vxsym:Pr:1}, 
\eqref{eq:vxsym:Pr:2} and $(\eps\mu)^2 \ge \Lambda$ imply 
\eqref{eq:vxsym}: 
\begin{align}
\label{eq:vxsym:Pr:T}
\Pr(R_{\cU} \le r) &\ge 1- \indic{\mu>0}\Lambda/(\Lambda+(\eps \mu)^2),\\
\label{eq:vxsym:Pr:Z}
\Pr(Z_{\cU} \le z) &\ge \bigl(1-\indic{\mu>0}\Lambda/(\Lambda+2(\eps \mu)^2)\bigr)4c\Pr(Y_G \le (1-\lambda\eps) \E Y_G). 
\end{align}

We note first that in the trivial case $\mu = 0$,   
almost surely $X=0$ and thus $Z=0$ which
implies $R_{\cU}=Z_{\cU}= 0$; hence also $z=0$ and $r=0$ 
so that \eqref{eq:vxsym:Pr:T}--\eqref{eq:vxsym:Pr:Z} follow trivially. 
We may thus assume $\mu>0$.

We next estimate $\E Z_{\cU}$. Let $\fX \subseteq [n]$ with 
$|\fX|=|\cU|$ be chosen uniformly at random, and independent of 
$\Gamma_{\vp}$. With the definitions of $\cH(\cdot,\beta)$ and 
$\cG(\cdot)$ in mind, using linearity of expectation we deduce 
\begin{equation}\label{eq:vxsym:E:fX}
\E(Z_{\fX} \mid \Gamma_{\vp}) = \sum_{(\alpha,\beta) \in \cH \times \cG} \indic{\beta \subseteq \alpha}\Pr(V(\beta) \subseteq \fX \text{ and } V(\alpha) \setminus V(\beta) \subseteq [n] \setminus \fX) I_{\alpha} ,
\end{equation}
where the measure $\Pr$ is with respect to the (random) choice of 
$\fX$. Note that, whenever $\beta \subseteq \alpha$, we have 
\begin{equation*}
\begin{split}
&\sigma_{\alpha,\beta} = \Pr(V(\beta) \subseteq \fX \text{ and } V(\alpha) \setminus V(\beta) \subseteq [n] \setminus \fX) = \frac{\binom{n-v_H}{|\cU|-v_G}}{\binom{n}{|\cU|}} .
\end{split}
\end{equation*}
Recall that $\bigl||\cU|-n/2\bigr| \le \ell$. For fixed $\ell$, $v_G$ 
and $v_H$ a short calculation shows that $\sigma_{\alpha,\beta} \to 2^{-v_H}$ 
as $n \to \infty$, so that $\sigma_{\alpha,\beta} \ge 2^{-(v_H+1)}= 4\lambda^{-1}$ 
for $n \ge n_0(H,\ell)$. Using \eqref{eq:vxsym:E:fX} and the definition 
of $Z$ we infer $\E(Z_{\fX} \mid \Gamma_{\vp}) \ge 4\lambda^{-1} Z$, so 
that $\E(Z_{\fX}) \ge 4\lambda^{-1} \E Z$. By definition, we have 
$\E(Z_{\fX} \mid \fX=S) = \E Z_{S}$ for all $S \subseteq [n]$ with 
$|S|=|\cU|$. Since $\E Z_S = \E Z_{\cU}$ by symmetry, we infer 
$\E Z_{\fX} = \E Z_{\cU}$, so that 
\begin{equation}\label{eq:vxsym:E:ZU}
\E Z_{\cU} \ge 4\lambda^{-1} \E Z. 
\end{equation}

Turning to \eqref{eq:vxsym:Pr:T}, note that $R_{\cU}$ is a restriction 
of $Z$ to a subset of all pairs $(\alpha,\beta) \in \cH \times \cG$. As 
Harris' inequality implies $\E(I_{\alpha_1}I_{\alpha_2}) \ge \E I_{\alpha_1} \E I_{\alpha_2}$, 
it follows that $\Var R_{\cU} \le \Var Z = \tau^2 \Var X_H \le \tau^2 \Lambda$. 
Recalling $\E R_{\cU} = \E Z - \E Z_{\cU}$ and the definitions of $r$ 
and $z$, using \eqref{eq:vxsym:E:ZU} we have 
$r-\E R_{\cU} = (\eps\lambda/2) \E Z_{\cU}-\eps\E Z \ge \eps \E Z = \tau \eps \mu$. 
So, if $\mu>0$, then the one-sided Chebyshev's inequality (Claim~\ref{cl:Ch}) 
yields 
\[
\Pr(R_{\cU} > r) \le \Pr(R_{\cU} \ge \E R_{\cU}+\tau \eps \mu) 
\le \tau^2\Lambda/(\tau^2\Lambda+(\tau\eps \mu)^2)
= \Lambda/(\Lambda+(\eps \mu)^2).
\]

In the remainder we focus on \eqref{eq:vxsym:Pr:Z}. Observing that $Y_G =\sum_{\beta \in \cG(\cU)} I_{\beta}$, 
we denote by $\cE$ the event that $Y_G \le (1-\lambda\eps) \E Y_G$ holds. 
With foresight, we define $X_{\beta} = \sum_{\alpha \in \cH(\cU,\beta)} I_{\alpha \setminus \beta}$ 
and $X_{\beta_1,\beta_2} =\sum_{(\alpha_1,\alpha_2) \in \cH(\beta_1,\beta_2)} I_{(\alpha_1 \cup \alpha_2) \setminus (\beta_1 \cup \beta_2)}$, 
where 
\[
\cH(\beta_1,\beta_2) =\bigl\{(\alpha_1,\alpha_2) \in \cH(\cU,\beta_1) \times \cH(\cU,\beta_2): \bigl[Q(\alpha_1) \cap Q(\alpha_2)\bigr]\setminus \bigl[Q(\beta_1) \cup Q(\beta_2)\bigr] \neq \emptyset\bigr\} .
\]
Let $\cF$ be the family of all pairwise non-isomorphic graphs that are 
unions of two (not necessarily distinct) copies of $G$. The point is 
that $\cF$ naturally defines a partition $(\cP_F)_{F \in \cF}$ of the 
set of all pairs of graphs $(\beta_1,\beta_2) \in \cG(\cU) \times \cG(\cU)$ 
with $\cH(\beta_1,\beta_2) \neq \emptyset$ (as each $\beta_1 \cup \beta_2$ 
is isomorphic to some $F \in \cF$). Furthermore, since every $F \in \cF$ 
satisfies $v_G \le v_F \le 2 v_G$, we have, say, 
$|\cF| \le 2^{\binom{2 v_G}{2}} \cdot 2^{v_G} \le 4^{v_G^2}$. 
Let $\Psi_{F} =\sum_{(\beta_1,\beta_2) \in \cP_F}I_{\beta_1 \cup \beta_2}$, 
and define $\cD$ as the event that $\Psi_{F} \le 2 \E \Psi_{F}$ for 
all $F \in \cF$. Using Harris' inequality and Markov's inequality, we 
deduce 
\begin{equation}\label{eq:vxsym:D}
\Pr(\cE \cap \cD) \ge \Pr(\cE) \prod_{F \in \cF} \Pr(\Psi_{F} \le 2 \E \Psi_{F}) \ge 2^{-|\cF|} \Pr(\cE) \ge 4c \Pr(\cE).
\end{equation}
For brevity, we write $\Pr^*$ for the conditional measure with respect 
to the status of all edges in $G^{(k)}_{n,p}[\cU]$. We use $\E^*$ and 
$\Var^*$ analogously. Since $\cE \cap \cD$ is determined by 
$E(G^{(k)}_{n,p}[\cU])$, we have 
\begin{equation}\label{eq:vxsym:Pr:Z2}
\Pr(Z_{\cU} \le z) \ge \Pr(\{Z_{\cU} \le z\} \cap \cE \cap \cD) = \E\bigl(\Pr^*(Z_{\cU} \le z) \indic{\cE \cap \cD} \bigr) .
\end{equation}
In the following we estimate $\Pr^*(Z_{\cU} \le z)$ whenever 
$\cE \cap \cD$ holds. Recall that for all $\beta \in \cG(\cU)$ and 
$\alpha \in \cH(\cU,\beta)$ we have $\beta \subseteq \alpha$, 
$V(\beta) \subseteq \cU$ and $V(\alpha) \setminus V(\beta) \subseteq [n] \setminus \cU$. 
Since every copy of $G$ in $H$ is induced, for all $f \in Q(\alpha) \setminus Q(\beta)$ 
we infer $f \not\in E(K^{(k)}_{n}[\cU])$. 
Using $Q(\beta) \subseteq Q(\alpha)$ 
it follows that $\E^* I_{\alpha}=I_{\beta}\E^* I_{\alpha \setminus \beta} = I_{\beta}\E I_{\alpha \setminus \beta}$. 
By symmetry, $\E X_{\beta}$ is independent of the choice of 
$\beta \in \cG(U)$, and so $\E^* Z_{\cU} = \sum_{\beta \in \cG(\cU)}I_{\beta} \E X_{\beta} = Y_G \E X_{\tilde{\beta}}$ 
for any $\tilde{\beta} \in \cG(\cU)$. Taking expectations, we deduce 
$\E Z_{\cU} = \E Y_G \E X_{\tilde{\beta}}$. Consequently $\E^* Z_{\cU} \le (1-\lambda\eps) \E Z_{\cU}$ 
whenever $\cE$ holds, in which case, using the definition of $z$ and 
\eqref{eq:vxsym:E:ZU}, we have 
\begin{equation}\label{eq:vxsym:z}
z - \E^* Z_{\cU} \ge (\eps\lambda/2) \E Z_{\cU} \ge 2 \eps \E Z = 2 \tau \eps \mu .
\end{equation} 
Turning to the conditional variance of $Z_{\cU}$, note that, by 
symmetry (analogous as for $Z$), we have 
\begin{equation}\label{eq:vxsym:Delta}
\begin{split}
\tau^2 \Lambda &= \sum_{\alpha \in \cH} \sum_{\substack{(\beta_1,\beta_2) \in \cG\times\cG:\\ \beta_1 \subseteq \alpha, \beta_2 \subseteq \alpha}} \E I_{\alpha} + \sum_{(\alpha_1,\alpha_2) \in \cH\times\cH: \alpha_1 \sim \alpha_2} \sum_{\substack{(\beta_1,\beta_2) \in \cG\times\cG:\\ \beta_1 \subseteq \alpha_1, \beta_2 \subseteq \alpha_2}} \E I_{\alpha_1 \cup \alpha_2}\\
& = \sum_{(\beta_1,\beta_2) \in \cG \times \cG} \E I_{\beta_1 \cup \beta_2} \sum_{\substack{(\alpha_1,\alpha_2) \in \cH\times\cH: \beta_1 \subseteq \alpha_1, \beta_2 \subseteq \alpha_2, \\ Q(\alpha_1) \cap Q(\alpha_2) \neq \emptyset}} \E I_{(\alpha_1 \cup \alpha_2) \setminus (\beta_1 \cup \beta_2)} .
\end{split}
\end{equation}
As before, $\E^* I_{\alpha_1 \cup \alpha_2} = I_{\beta_1 \cup \beta_2} \E^*I_{(\alpha_1 \cup \alpha_2)\setminus(\beta_1 \cup \beta_2)} = I_{\beta_1 \cup \beta_2} \E I_{(\alpha_1 \cup \alpha_2)\setminus(\beta_1 \cup \beta_2)}$ 
for all $(\beta_1,\beta_2) \in \cG(U)\times\cG(U)$ and $(\alpha_1,\alpha_2) \in \cH(U,\beta_1) \times \cH(U,\beta_2)$. 
It follows that 
\begin{equation*}\label{eq:vxsym:ZUVar:1}
\Var^* Z_{\cU} \le \sum_{(\beta_1,\beta_2) \in \cG(U)\times\cG(U)} I_{\beta_1 \cup \beta_2} \sum_{\substack{(\alpha_1,\alpha_2) \in \cH(U,\beta_1) \times \cH(U,\beta_2):\\ [Q(\alpha_1) \cap Q(\alpha_2)] \setminus [Q(\beta_1) \cup Q(\beta_2)] \neq \emptyset}} \E I_{(\alpha_1 \cup \alpha_2) \setminus (\beta_1 \cup \beta_2)} .
\end{equation*}
Now, recalling the definitions of $\cH(\beta_1,\beta_2)$, 
$X_{\beta_1,\beta_2}$, $\cF$ and $\Psi_{F}$, we infer 
\begin{equation*}\label{eq:vxsym:ZUVar:2}
\Var^* Z_{\cU} \le \sum_{F \in \cF}\sum_{(\beta_1,\beta_2) \in \cP_F} I_{\beta_1 \cup \beta_2} \E X_{\beta_1,\beta_2} \le \sum_{F \in \cF} \Psi_{F} \max_{(\beta_1,\beta_2)\in \cP_F}\E X_{\beta_1,\beta_2} .
\end{equation*}
By symmetry, we have $\max_{(\beta_1,\beta_2)\in \cP_F}\E X_{\beta_1,\beta_2} = \min_{(\beta_1,\beta_2)\in \cP_F}\E X_{\beta_1,\beta_2}$ 
for all $F \in \cF$. So, with analogous considerations as above, 
whenever $\cD$ holds we have 
\begin{equation}\label{eq:vxsym:ZUVar:3}
\begin{split}
\Var^* Z_{\cU} & \le 2\sum_{F \in \cF} \E \Psi_{F} \min_{(\beta_1,\beta_2)\in \cP_F}\E X_{\beta_1,\beta_2} = 2 \sum_{F \in \cF}\sum_{(\beta_1,\beta_2) \in \cP_F} \E I_{\beta_1 \cup \beta_2} \E X_{\beta_1,\beta_2} \\
&= 2 \sum_{(\beta_1,\beta_2) \in \cG(U)\times\cG(U)} \E I_{\beta_1 \cup \beta_2} \sum_{(\alpha_1,\alpha_2) \in \cH(\beta_1,\beta_2)} \E I_{(\alpha_1 \cup \alpha_1) \setminus (\beta_1 \cup \beta_2)} \le 2 \tau^2 \Lambda ,
\end{split}
\end{equation}
where the last inequality follows by comparison with \eqref{eq:vxsym:Delta}. 
If $\mu>0$, then, using \eqref{eq:vxsym:z}, the one-sided 
Chebyshev's inequality (Claim~\ref{cl:Ch}) 
and \eqref{eq:vxsym:ZUVar:3},  
whenever $\cE \cap \cD$ holds we have 
\begin{equation}\label{eq:vxsym:Pr:ZU}
\Pr^*(Z_{\cU} > z) \le \Pr^*(Z_{\cU} \ge \E^* Z_{\cU}+2 \tau \eps \mu) 
\le2\tau^2 \Lambda/(2\tau^2\Lambda+(2\tau\eps \mu)^2) 
= \Lambda/(\Lambda+2(\eps \mu)^2) .
\end{equation}
Inserting \eqref{eq:vxsym:Pr:ZU} into
\eqref{eq:vxsym:Pr:Z2}, we infer (for $\mu>0$) 
\[
\Pr(Z_{\cU} \le z) \ge \bigl(1-\Lambda/(\Lambda+2(\eps \mu)^2)\bigr) \Pr(\cE \cap \cD) ,
\]
which together with \eqref{eq:vxsym:D} implies \eqref{eq:vxsym:Pr:Z} 
by definition of $\cE$. 
\end{proof}
A variant of the proof applies to \emph{rooted} copies of $H$, see, 
e.g., Section~3 in~\cite{UTRSG} for a precise definition. 
The basic idea is to map the vertex set of the root $R$ to $[r]$, and 
the remaining vertices of $G$ and $H$ to $\cU \subseteq [n] \setminus [r]$ 
and $[n] \setminus (\cU \cup [r])$, respectively; 
we leave the details to the interested reader.

\section{Applications}\label{sec:ex}
In this section we illustrate the bootstrapping approaches of 
Section~\ref{sec:bstrp} via pivotal examples from additive and 
probabilistic combinatorics. In Section~\ref{sec:RS} we consider the 
lower tail of the number of arithmetic progressions (and Schur 
triples) in random subsets of the integers. In Section~\ref{sec:RHG} 
we then turn to our main example: the lower tail of subgraph 
counts in random hypergraphs.

\subsection{Random subsets of the integers}\label{sec:RS}
Let $X_k=X_k(n,p)$ be the number of arithmetic progressions of length 
$k \ge 2$ in the binomial random subset $\Gamma_{\vp}$ of the integers 
$\Gamma=[n]=\{1, \ldots, n\}$, where $\vp=(p, \ldots, p)$. 
Note that $\E X_k = \Theta(n^2p^k)$; see also Section~\ref{sec:bstrp:sdcmp}. 
The following theorem gives fair exponential bounds for the lower 
tail of $X_k$, and its proof closely follows the strategy outlined 
in Section~\ref{sec:bstrp}. 
\begin{theorem}\label{thm:AP}
Given $k \ge 2$, let $\Psi_k=\Psi_k(n,p) = \min\{n^2p^k,np\}$. 
There are positive constants $c$, $C$, $D$ and $n_0$, all depending only on $k$, such that for all $n \ge n_0$, $p \in [0,1)$ and $\eps \in (0,1]$ satisfying $\eps^2 \Psi_k \ge \indic{\eps < 1}D$ we have
\begin{equation}\label{eq:AP}
\exp\bigl\{- (1-p)^{-5} C \eps^2 \Psi_k\bigr\} \le \Pr(X_{k} \le (1-\eps) \E X_{k}) \le \exp\bigl\{-c \eps^2 \Psi_k\bigr\}.
\end{equation}
\end{theorem}
\begin{proof}
Let $\mu=\E X_k$, $\Lambda=\Lambda(X_k)$ and $\delta = \delta(X_k)$. 
Routine calculations, analogous to Example~3.2 in~\cite{JLR}, reveal 
that 
\begin{equation}\label{eq:AP:dL}
\delta = \Theta(n p^{k-1}+p) \quad \text{ and } \quad \mu^2/\Lambda = \mu/(1+\delta) = \Theta(\Psi_k),
\end{equation}
where the implicit constants depend only on $k$. Hence the upper 
bound of \eqref{eq:AP} is an immediate consequence of \eqref{eq:J}. 
For the lower bound we pick, with foresight, $D=D(k) \ge 1$ such that 
$\E X_k \ge \Psi_k/D$ and $\mu^2/\Lambda \ge \Psi_k/D$ for $n \ge n_0(k)$. 

If $\Psi_k = n^2p^k$, then Theorem~\ref{thm:LT2} (with $X=X_k$) yields 
\[
\Pr(X_{k} \le (1-\eps) \E X_{k}) \ge \exp\bigl\{- \Theta((1-p)^{-5}\eps^2 \Psi_k)\bigr\} 
\]
since $\eps^2 \E X_k \ge \eps^2 \Psi_k/D \ge \indic{\eps<1}$, 
$\Pi(X_k)=p^k \le p$, $\delta=O(1)$ and $\E X_k = \Theta(\Psi_k)$. 

If $\Psi_k = np$, then Theorem~\ref{thm:rsize} (with $X=X_k$) and 
Theorem~\ref{thm:LT2} (with $X=|\Gamma_{\vp}|$) yield, with $d=1/2+ \indic{\eps=1}1/2$,
\[
\Pr(X_{k} \le (1-\eps) \E X_{k}) \ge d \Pr(|\Gamma_{\vp}| \le (1-\eps) \E |\Gamma_{\vp}|) \ge \exp\bigl\{-\indic{\eps<1}\log 2-\Theta((1-p)^{-5} \eps^2 \Psi_k)\bigr\} 
\]
since $(\eps \mu)^2 \ge \Lambda \eps^2 \Psi_k/D \ge \indic{\eps<1}\Lambda$, 
$\eps^2 \E |\Gamma_{\vp}| = \eps^2 \Psi_k \ge \indic{\eps<1}$ and 
$\E |\Gamma_{\vp}|=\Psi_k$. This completes the proof of \eqref{eq:AP} since 
$\indic{\eps<1}\log 2 \le \indic{\eps<1}D \le (1-p)^{-5} \eps^2 \Psi_k$.
\end{proof}
For Schur triples, which are defined in Section~\ref{sec:bstrp:sdcmp}, 
the same calculations carry over (with $k=3$; the point is that 
\eqref{eq:AP:dL} holds), yielding an analogous lower tail estimate. 
Related results for the upper tail of arithmetic progressions and 
Schur triples have been established by Warnke~\cite{AP}.

\subsection{Random hypergraphs}\label{sec:RHG} 
Finally, we consider the lower tail of the number $X_H=X_H(n,p)$ of 
copies of a given $k$-graph $H$ in $G^{(k)}_{n,p}$, and prove 
Theorems~\ref{thm:HG:const}--\ref{thm:HG:extr}. Here the following 
precise analysis of $\Lambda(X_H)$ is at the heart of our approach. 
In fact, Lemma~\ref{lem:lambda:HG} is essentially given in~\cite{JLR_ineq} 
(for $k=2$), but the restriction to subgraphs from $\cI_H$ is new and 
crucial for our purposes: the key point is that \emph{every} copy of 
$G \in \cI_H$ in $H$ is induced. 
Recall that $m_k(H)$ is defined by \eqref{mk}.
\begin{lemma}\label{lem:lambda:HG}
Let $H$ be a $k$-graph with $e_H \ge 1$. Define $\cI_H$ as the 
collection of all non-isomorphic subgraphs $J \subseteq H$ which 
satisfy $e_J \ge \max\{e_K,1\}$ for all $K \subseteq H$ with 
$v_K=v_J$. For all $p=p(n) \in (0,1]$ we have 
\begin{gather}
\label{eq:HG:lambda}
\Lambda(X_H) = (1+o(1))\sum_{J \in \cI_H}C^2_{J,H} \frac{(\E X_H)^2}{\E X_J} = \Theta\Bigl(\frac{(\E X_H)^2}{\min_{J \in \cI_H}\E X_J}\Bigr),\\
\label{eq:HG:min} 
\min_{J \in \cI_H}\E X_J = o(\min_{J \subseteq H, e_J \ge 1, J {\not\in}^* \cI_H}\E X_J ),
\end{gather}
where $C_{J,H}$ denotes the number of copies of $J$ in $H$, and 
$J {\not\in}^* \cI_H$ means that there is no $J' \in \cI_H$ which 
is isomorphic to $J$. In addition, $p = \omega(n^{-1/m_k(H)})$ implies 
$\min_{J \in \cI_H}\E X_J = \binom{n}{k}p$ and 
$\Lambda(X_H) = (1+o(1)) e_H^2 (\E X_H)^2/[\binom{n}{k}p]$. 
\end{lemma}
The fairly standard proof of Lemma~\ref{lem:lambda:HG} is deferred to Appendix~\ref{apx}.  
In the following proofs of Theorems~\ref{thm:HG:const}--\ref{thm:HG:extr} 
we shall not explicitly discuss the upper bounds: once the form of 
$(\E X_H)^2/\Lambda(X_H)$ has been established, these are immediate 
consequences of \eqref{eq:J}. 
\begin{proof}[Proof of Theorem~\ref{thm:HG:const}]
Let $d=2^{-(4^{v_H^2}+2)}$, $\lambda = 2^{v_H+3}$ and $\eps_0=(2\lambda)^{-1}$. 
Since the claim is trivial otherwise, we henceforth assume $p > 0$. 
Furthermore, we use the convention that all implicit constants depend 
only on $H$, and tacitly assume $n \ge n_0(H)$ whenever necessary. 
Suppose that $\Phi_H = \E X_G$ for $G \subseteq H$ with $e_G \ge 1$. 
Using \eqref{eq:HG:lambda} and \eqref{eq:HG:min} we infer $G \in \cI_H$, 
$(\E X_H)^2/\Lambda(X_H) = \Theta(\Phi_H)$ and $\delta(X_G) = O(1)$. 
With foresight, we pick $D=D(H) \ge \log(1/d)$ such that 
$(\E X_H)^2/\Lambda(X_H) \ge \Phi_H/D$ holds. 

If $\eps \in [\eps_0,1]$, then $\Pi(X_G) = p^{e_G} \le p$, 
$\E X_G=\Phi_H$, $1 \le \eps_0^{-2} \eps^2$ and \eqref{eq:H} yield 
\begin{equation}\label{eq:HG:const:H}
\Pr(X_H \le (1-\eps) \E X_H) \ge \Pr(X_H = 0) \ge \Pr(X_G=0) \ge \exp\bigl\{-(1-p)^{-1}\eps_0^{-2} \eps^2 \Phi_H\bigr\} .
\end{equation}
 
It remains to establish \eqref{eq:HG:const} when $\eps < \eps_0$. 
We shall eventually apply Theorem~\ref{thm:vxsym} with 
$U=[\floor{n/2}]$, where $Y_G$ counts the total number of copies of 
$G$ whose vertex sets are completely contained in $U$. 
Since $G_{n,p}^{(k)}[\cU]$ has the same distribution as 
$G_{n',p}^{(k)}$ with $n'=|\cU| \approx n/2$, we readily deduce 
$3^{-v_G} \E X_G \le \E Y_G \le \E X_G$ and $\delta(Y_G)=\Theta(\delta(X_G))$. 
Furthermore, $G \in \cI_H$ implies that every copy of $G$ in $H$ is 
induced. So, using $\lambda\eps \le 1/2$, $\Pi(Y_G) \le p$, 
$\delta(Y_G)=O(1)$ and $\E Y_G = \Theta(\Phi_H)$, a combination of 
Theorem~\ref{thm:vxsym} and Theorem~\ref{thm:LT2} yields 
\[
\Pr(X_H \le (1-\eps) \E X_H) \ge d \Pr(Y_G \le (1-\lambda\eps) \E Y_G) \ge \exp\bigl\{-\log(1/d)-\Theta((1-p)^{-5} \lambda^2 \eps^2 \Phi_H)\bigr\}
\]
since $\eps^2 (\E X_H)^2 \ge \eps^2 \Phi_H \Lambda(X_H)/D \ge \Lambda(X_H)$ 
and $(\lambda\eps)^2 \E Y_G \ge \lambda^2 3^{-v_G} \eps^2 \E X_G \ge \eps^2 \Phi_H \ge D \ge 1$. 
This completes the proof of \eqref{eq:HG:const} since 
$\log(1/d) \le D \le (1-p)^{-5} \eps^2 \Phi_H$. 
\end{proof}
\begin{proof}[Proof of Theorem~\ref{thm:HG:ULO}]
Since the claim is trivial otherwise, we henceforth assume $p > 0$. 
Furthermore, since $p=o(1)$ we have $\Pi=o(1)$. 
Recalling the properties of $G$, using \eqref{eq:HG:lambda} and 
\eqref{eq:HG:min} we infer $G \in \cI_H$, 
$(\E X_H)^2/\Lambda(X_H) = (1+o(1)) \E X_G$ and $\delta(X_G) = o(1)$. 

In the special case $e_G=1$, note that uniqueness of $G$ in $H$ 
implies $e_H=1$, and that minimality of $\E X_G$ implies $v_G=k$. 
Thus $X_H=X_G\binom{n-k}{v_H-k}$ and $\delta(X_G)=0$. Using 
$\Pr(X_H \le (1-\eps) \E X_H) = \Pr(X_G \le (1-\eps) \E X_G)$, the 
lower bound of \eqref{eq:HG:ULO} now follows from Theorem~\ref{thm:LT} 
(applied to $X_G$), where $\xi=o(1)$ by our assumptions. 

Henceforth we thus assume $e_G \ge 2$. Now, in case of $H=G$ the 
lower bound of \eqref{eq:HG:ULO} follows directly from 
Theorem~\ref{thm:LT}. 
In the main case, where $G \subsetneq H$ and $e_G \ge 2$, there 
exists, by assumption, $\omega = \omega(n) \to \infty$ such that 
$\eps^2 \E X_G \ge \indic{\eps<1}\omega \log(e/\eps)$. Setting 
$\gamma = 2\exp\{-\omega^{1/2}\} = o(1)$ we have 
(when $\omega\ge1$) 
$\eps^2 \E X_G \ge \indic{\eps<1} \omega^{1/2} \log(2/(\gamma \eps))$, 
which together with Lemma~\ref{lem:varphi} yields 
$2^{-1}\gamma \eps \ge \indic{\eps<1} \exp\{-2\omega^{-1/2}\varphi(-\eps) \E X_G\}$. 
So, if $(1+\gamma)\eps < 1$ and $3 \sqrt{\gamma} < 1-\eps$, then a 
combination of Theorem~\ref{thm:rcor} (with $X=X_H$, $Y=X_G$ and 
$\kappa=0$), Theorem~\ref{thm:LT} (for $X_G$) 
and Lemma~\ref{lem:varphi3} (with 
$A=1+\gamma$) establishes \eqref{eq:HG:ULO}. Otherwise 
$\eps \ge 1- \max\{\gamma/(1+\gamma),3\sqrt{\gamma}\} = 1-o(1)$ holds, 
and then a combination of \eqref{eq:rcor:0} (with $X=X_H$ and $Y=X_G$) 
and Lemma~\ref{lem:LT4}  (for $X_G$) 
completes the proof. 
\end{proof}
\begin{proof}[Proof of Theorem~\ref{thm:HG:extr}]
We start with the main case $\eps=o(1)$. Note that Lemma~\ref{lem:lambda:HG} 
implies $\min_{J \in \cI_H}\E X_J = \binom{n}{k}p=\E |\Gamma_{\vp}|$ 
and $(\E X_H)^2/\Lambda(X_H) = (1+o(1)) \E |\Gamma_{\vp}|/e_H^2$. 
By assumption, there is $\omega=\omega(n) \to \infty$ such that 
$\eps \le 1/\omega$ and $\eps^2\binom{n}{k}p \ge \omega$. 
Let $\tau=6 e_H \omega^{-1/2}=o(1)$ 
and $A=(1+\tau)/e_H$, so that 
$\varphi(-A\eps) \le (1+o(1)) \varphi(-\eps)/e_H^2$ by 
Lemma~\ref{lem:varphi3}. Since $p=o(1)$, a combination of 
Theorem~\ref{thm:rsize2} (with $X=X_H$ and $k=e_H$) and 
Theorem~\ref{thm:LT} (with $X=|\Gamma_{\vp}|$) establishes 
\eqref{eq:HG:extr}, where the factor $c=1/2$ is negligible due 
to $\varphi(-\eps)\binom{n}{k}p \to \infty$. 

The remaining $\eps=1-o(1)$ estimate of \eqref{eq:HG:extr} 
follows from Lemma~\ref{lem:G:zero}  below and Lemma~\ref{lem:varphi2} 
since $1-p=e^{-(1+o(1))p}$ and $\varphi(-\eps)=1+o(1)$ for 
$p=o(1)$ and $\eps=1-o(1)$, respectively. 
\end{proof} 

The proof above used the following lemma, 
which follows from results of Saxton and Thomason~\cite{ST}.
\begin{lemma}\label{lem:G:zero}
Let $H$ be a $k$-graph with $e_H \ge 1$. If $p=p(n) \in [0,1]$ 
and $\eps=\eps(n) \in (0,1]$ satisfy $p= \omega(n^{-1/m_k(H)})$ and 
$\eps=1-o(1)$, then we have
\begin{equation}\label{eq:G:zero}
\Pr(X_H \le (1-\eps) \E X_H) = (1-p)^{(1+o(1)) (1-\pi_H)\binom{n}{k}} .
\end{equation}
\end{lemma}
\begin{proof}
For the lower bound, let $\cT_{n,H}$ be any hypergraph which achieves 
equality in the definition of $\ex(n,H)$. As every subgraph of 
$\cT_{n,H}$ is $H$-free, it follows that 
\begin{equation*}
\Pr(X_H \le (1-\eps) \E X_H) \ge \Pr(X_H = 0) \ge \Pr(G^{(k)}_{n,p} \subseteq \cT_{n,H}) = (1-p)^{\binom{n}{k}-e(\cT_{n,H})} .
\end{equation*}
This establishes the lower bound of~\eqref{eq:G:zero} since 
$e(\cT_{n,H}) = (\pi_H+o(1))\binom{n}{k}$ and $1-\pi_H \in (0,1]$.

Turning to the corresponding upper bound, 
we first consider the case $e_H \ge 2$. 
Let $0 < \delta \le (1-\pi_H)/3$. 
Theorem 9.2 in~\cite{ST} implies that there is $c=c(H,\delta) >0$ such 
that for $n \ge c$ the following holds for all $q \in [n^{-1/m_k(H)},1/c]$: 
there exists $s\le c$ and a mapping $T\mapsto C(T)$ of sequences
$T=(T_1,\dots,T_s)$ with $T_i\subseteq\eknk$ to sets $C(T)\subseteq\eknk$
such that
for every $k$-graph $G$ on $n$ vertices with less than $n^{v_H}q^{e_H}$ 
copies of $H$ there exists $T=(T_1, \ldots,T_s)$ 
such that
$E(G) \subseteq C(T)$,  $|C(T)| \le (\pi_H+\delta)\binom{n}{k}=F$
and further
$\sum_{1 \le i \le s}|T_i| \le cqn^{k}=U$ and 
$\bigcup_{1 \le i \le s}T_i\subseteq E(G)$.
(Recall that $\eknk$ is the set of all edges in the complete 
$k$-graph $\knk$. The mapping $T\mapsto C(T)$ is quite complicated; the
point of it is that we can bound the number of 'containers' $C(T)$ by the
number of sequences $T$.)

By assumption we have
$1-\eps \le 1/\omega$ and $p \ge \omega n^{-1/m_k(H)}$, 
where $\omega=\omega(n) \to \infty$.
Let $q=\omega^{-1/e_H}p$, so that 
$(1-\eps) \E X_H < \omega^{-1} n^{v_H}p^{e_H} = n^{v_H}q^{e_H}$ and 
$n^{-1/m_k(H)} \le q \le \omega^{-1/e_H} \le 1/c$ for $n \ge n_0(c)$. 
Note that we can construct a superset of all possible $T=(T_1, \ldots,T_s)$ as
follows: we first decide on 
$|\bigcup_{1 \le i \le s} T_i|=u$, 
then select $u$ edges of $K^{(k)}_n$ and decide on all the $T_i$ in 
which they appear. So, taking the union bound over all choices of $T$
that are possible for $G=G^{(k)}_{n,p}$,
using $\bigcup_{1 \le i \le s}T_i \subseteq E(G^{(k)}_{n,p})$ and 
$E(G^{(k)}_{n,p}) \setminus C(T)=\emptyset$ it follows that 
\begin{equation}\label{eq:G:zero:UB1}
\begin{split}
 \Pr(X_H \le (1-\eps) \E X_H) 
& \le 
\sum_{0 \le u \le U} \binom{\binom{n}{k}}{u}(2^s)^up^u(1-p)^{\binom{n}{k}-F}
. 
\end{split}
\end{equation}
Hence, recalling the definitions of $F$ and $U$, for any $\theta\in(0,1]$ we obtain 
\begin{equation}\label{eq:G:zero:2}
\begin{split}
 \Pr(X_H \le (1-\eps) \E X_H) 
& \le
(1-p)^{\binom{n}{k}-F}
\theta^{-U}
\sum_{0 \le u \le U} \binom{\binom{n}{k}}{u}(2^s)^up^u\theta^u
\\& 
\le
(1-p)^{\binom{n}{k}-F}
\theta^{-U}
\bigpar{1+2^sp\theta}^{\binom{n}{k}}
\le
 (1-p)^{(1-\pi_H-\delta)\binom{n}{k}}
e^{cqn^k\log(1/\theta)+2^sp\theta\binom{n}{k}}
.
\end{split}
\end{equation}
Choose $\theta=q/p=o(1)$. Then $q\log(1/\theta)=p\theta\log(1/\theta)=o(p)$, 
$e^{p} \le (1-p)^{-1}$ and \eqref{eq:G:zero:2} yield, for $n \ge n_0(c,s,\delta)$, 
\begin{equation}\label{eq:G:zero:3}
\begin{split}
 \Pr(X_H \le (1-\eps) \E X_H) 
& \le
 (1-p)^{(1-\pi_H-\delta)\binom{n}{k}}
e^{o\lrpar{p\binom{n}{k}}}
\le
 (1-p)^{(1-\pi_H-2\delta)\binom{n}{k}}
.
\end{split}
\end{equation}
It follows as usual that there is some $\delta(n)\to 0$ such that 
\eqref{eq:G:zero:3} holds with $\delta=\delta(n)$ for $n \ge n_0$, 
which together with $1-\pi_H \in (0,1]$ establishes the upper bound 
of~\eqref{eq:G:zero} when $e_H \ge 2$. 

Finally, in the remaining case $e_H=1$ (where Theorem 9.2 in~\cite{ST} 
does not apply) we have $X_H=e(G^{(k)}_{n,p})\binom{n-k}{v_H-k}$. 
Hence $X_H \le (1-\eps) \E X_H$ is equivalent to
$e(G^{(k)}_{n,p}) \le (1-\eps) \binom{n}{k}p$. Since 
$e(G^{(k)}_{n,p}) \sim \Bin\bigpar{\binom nk,p}$ and $\pi_H=0$, 
\eqref{eq:G:zero} follows by standard calculations. (For example, 
\eqref{eq:G:zero:UB1} holds with $s=0$ and $U=F=(1-\eps)p\binom{n}{k}$, 
and the reasoning of \eqref{eq:G:zero:2}--\eqref{eq:G:zero:3} carries 
over since $F=o(p\binom{n}{k})$ and $U \le \omega^{-1}pn^k = qn^k$.) 
\end{proof}

\begin{ack}
We would like to thank Andrew Thomason for giving us a 
draft of~\cite{ST} together with helpful comments on it. 
\end{ack}

\small
\begin{spacing}{0.9}
\bibliographystyle{plain}

\begin{thebibliography}{10}

\bibitem{AS}
N.~Alon and J.~Spencer.
\newblock {\em The probabilistic method}. Third edition. 
\newblock {\em Wiley-Interscience Series in Discrete Mathematics and Optimization}.
\newblock John Wiley \& Sons Inc., Hoboken, NJ (2008). 

\bibitem{B82}
A.D.~Barbour.
\newblock Poisson convergence and random graphs.
\newblock {\em Math.\ Proc.\ Cambridge Philos.\ Soc.} {\bf 92} (1982), 349--359.

\bibitem{BB1981}
B.~Bollob{\'a}s.
\newblock Threshold functions for small subgraphs.
\newblock {\em Math.\ Proc.\ Cambridge Philos.\ Soc.} {\bf 90} (1981), 197--206.

\bibitem{BR1998}
B.~Bollob{\'a}s and O.~Riordan.
\newblock Colorings generated by monotone properties.
\newblock {\em Random Struct.\ Alg.} {\bf 12} (1998), 1--25.

\bibitem{K3TailCh}
S.~Chatterjee.
\newblock The missing log in large deviations for triangle counts.
\newblock {\em Random Struct.\ Alg.} {\bf 40} (2012), 437--451.

\bibitem{CD2014}
S.~Chatterjee and A.~Dembo.
\newblock Nonlinear large deviations.
\newblock Preprint, 2014. \texttt{arXiv:1401.3495}.

\bibitem{CV2011}
S.~Chatterjee and S.R.S.~Varadhan.
\newblock The large deviation principle for the {E}rd{\H o}s-{R}\'enyi random
 graph.
\newblock {\em European J.\ Combin.} {\bf 32} (2011), 1000--1017.

\bibitem{KkTailDK}
B.~DeMarco and J.~Kahn.
\newblock Tight upper tail bounds for cliques.
\newblock {\em Random Struct.\ Alg.} {\bf 41} (2012), 469--487.

\bibitem{DGL1996}
L.~Devroye, L.~Gy{\"o}rfi, and G.~Lugosi.
\newblock {\em A probabilistic theory of pattern recognition}. 
\newblock {\em Applications of Mathematics (New York)} {\bf 31}. 
\newblock Springer-Verlag, New York (1996).

\bibitem{ER1960}
P.~Erd{\H{o}}s and A.~R{\'e}nyi.
\newblock On the evolution of random graphs.
\newblock {\em Magyar Tud.\ Akad.\ Mat.\ Kutat\'o Int. K\"ozl.} {\bf 5} (1960), 17--61.

\bibitem{FKG}
C.M.~Fortuin, P.W.~Kasteleyn, and J.~Ginibre.
\newblock Correlation inequalities on some partially ordered sets.
\newblock {\em Comm.\ Math.\ Phys.} {\bf 22} (1971), 89--103.

\bibitem{Harris1960}
T.E.~Harris.
\newblock A lower bound for the critical probability in a certain percolation process.
\newblock {\em Proc.\ Cambridge Philos.\ Soc.} {\bf 56} (1960), 13--20.

\bibitem{Janson}
S.~Janson.
\newblock Poisson approximation for large deviations.
\newblock {\em Random Struct.\ Alg.} {\bf 1} (1990), 221--229.

\bibitem{Janson:Suen}
S.~Janson.
\newblock New versions of Suen's correlation inequality. 
\newblock {\em Random Struct.\ Alg.} {\bf 13} (1998), 467--483.

\bibitem{JLR_ineq}
S.~Janson, T.~{\L}uczak, and A.~Ruci{\'n}ski.
\newblock An exponential bound for the probability of nonexistence of a specified subgraph in a random graph.
\newblock In {\em Random graphs '87 ({P}ozna\'n, 1987)}, pp.~73--87, Wiley, Chichester (1990).

\bibitem{JLR}
S.~Janson, T.~{\L}uczak, and A.~Ruci{\'n}ski.
\newblock {\em Random graphs}.
\newblock {\em Wiley-Interscience Series in Discrete Mathematics and Optimization}. 
\newblock Wiley-Interscience, New York (2000).

\bibitem{UTSG}
S.~Janson, K.~Oleszkiewicz, and A.~Ruci{\'n}ski.
\newblock Upper tails for subgraph counts in random graphs.
\newblock {\em Israel J.\ Math.} {\bf 142} (2004), 61--92.

\bibitem{DL}
S.~Janson and A.~Ruci{\'n}ski.
\newblock The deletion method for upper tail estimates.
\newblock {\em Combinatorica} {\bf 24} (2004), 615--640.

\bibitem{UTRSG}
S.~Janson and A.~Ruci{\'n}ski.
\newblock Upper tails for counting objects in randomly induced subhypergraphs and rooted random graphs.
\newblock {\em Ark.\ Mat.} {\bf 49} (2011), 79--96.

\bibitem{KHG}
P.~Keevash.
\newblock Hypergraph Tur\'an problems. 
\newblock In {\em Surveys in combinatorics (Exeter 2011)}, pp.~83--139, Cambridge Univ.\ Press, Cambridge (2011).

\bibitem{KimVu2000}
J.H.~Kim and V.H.~Vu.
\newblock Concentration of multivariate polynomials and its applications.
\newblock {\em Combinatorica} {\bf 20} (2000), 417--434.

\bibitem{LZ2014}
E.~Lubetzky and Y.~Zhao.
\newblock On the variational problem for upper tails in sparse random graphs.
\newblock Preprint (2014). \texttt{arXiv:1402.6011}.

\bibitem{RiordanWarnke2012J}
O.~Riordan and L.~Warnke.
\newblock The {J}anson inequalities for general up-sets.
\newblock {\em Random Struct.\ Alg.}, to appear. \texttt{arXiv:1203.1024}.

\bibitem{R88}
A.~Ruci{\'n}ski.
\newblock When are small subgraphs of a random graph normally distributed?
\newblock {\em Probab.\ Theory Related Fields} {\bf 78} (1988), 1--10.

\bibitem{ST}
D.~Saxton and A.~Thomason.
\newblock Hypergraph containers.
(Revised version of arXiv:1204.6595v2.) 
\newblock Preprint (2014).

\bibitem{MS}
M.~{\v{S}}ileikis.
On the upper tail of counts of strictly balanced subgraphs.
\newblock {\em Electron.\ J.\ Combin.} {\bf 19} (2012), Paper 4. 

\bibitem{Spencer1990}
J.~Spencer.
\newblock Counting extensions.
\newblock {\em J.\ Combin.\ Theory Ser.\ A} {\bf 55} (1990). 247--255.

\bibitem{Suen}
W.-C.~Suen. 
A correlation inequality and a Poisson limit theorem for nonoverlapping balanced subgraphs of a random graph. 
\newblock {\em Random Struct.\ Alg.} {\bf 1} (1990), 231--242. 

\bibitem{Vu2001}
V.H.~Vu.
\newblock A large deviation result on the number of small subgraphs of a random graph.
\newblock {\em Combin.\ Probab.\ Comput.} {\bf 10} (2001), 79--94.

\bibitem{AP}
L.~Warnke.
\newblock Upper tails for arithmetic progressions in random subsets.
\newblock Preprint (2013).

\end{thebibliography}

\end{spacing}
\normalsize

\begin{appendix}
\section{Appendix}\label{apx}
In this appendix we prove Lemmas~\ref{lem:varphi}--\ref{lem:varphi3} 
and \ref{lem:lambda:HG}. 
\begin{proof}[Proof of Lemma~\ref{lem:varphi}]
By our conventions, \eqref{eq:varphi} is trivial for $\eps=1$, and so 
we henceforth assume $\eps \in [0,1)$. First, let 
$f(x)=2 \varphi(-x)-(1-x) \log^2 (1-x)$. Since $f'(x)=\log^2(1-x) \ge 0$ 
for $x \in [0,1)$, we infer $f(\eps) \ge f(0)=0$. Second, let 
$g(x) = 2 \varphi(-x)-x^2$. Since $1-x \le e^{-x}$ implies 
$g'(x)=-2\log(1-x)-2x \ge 0$ for $x \in [0,1)$, we infer 
$g(\eps) \ge g(0)=0$. Next, let $h(x) = \log^2(1-x)-2\varphi(-x)$. 
Since $h'(x) = -2x (1-x)^{-1}\log(1-x) \ge 0$ for $x \in [0,1)$, 
we infer $h(\eps) \ge h(0)=0$. Finally, $1-\eps \le e^{-\eps}$ 
implies $\varphi(-\eps) = (1-\eps)\log(1-\eps)+\eps \le \eps^2$. 
\end{proof}
\begin{proof}[Proof of Lemma~\ref{lem:varphi2}]
As \eqref{eq:varphi2} is trivial otherwise, we henceforth assume 
$\eps < 1$. Since $\varphi'(x)=\log(1+x) \le 0$ for $x \in [-1,0]$, 
we infer $\varphi(-\eps) \le \varphi(-1)=1$, which establishes the 
first inequality of \eqref{eq:varphi2}. 

Next, define $y=1-\eps$, and 
note that $y \in (0,e^{-1}]$. Let $g(x)=\phi(x-1)=1-x\log(e/x)$. 
Since 
$g'(x) = \log x \le 0$ for $x \in (0,1]$, we infer 
$g(y) \ge g(e^{-1}) = (e-2)/e > 0$. Let $h(x)=\sqrt{x}\log(e/x)$, 
and note that $h(y) > 0$. Since $h'(x) = -\log(ex)/(2 \sqrt{x}) \ge 0$ 
for $x \in (0,e^{-1}]$, we infer $h(y) \le h(e^{-1}) = 2/\sqrt{e}$. 
It follows that 
\[
\frac{1}{\varphi(-\eps)}-1 
=\frac{1-g(y)}{g(y)}
= \frac{\sqrt{y}h(y)}{g(y)} \le \frac{2 \sqrt{e y}}{e-2}\le 5 \sqrt{1-\eps},
\]
which establishes the second inequality of \eqref{eq:varphi2}. 
\end{proof}
\begin{proof}[Proof of Lemma~\ref{lem:varphi3}]
We first consider the case $y=A \eps \le 1$, so that $y \in [0,1]$. 
Since $\log(1-x) = -\sum_{j \ge 1} x^j/j \le -x -x^2/2$ for 
$x \in [0,1)$, we see that $\varphi(-y)= (1-y)\log(1-y)+y \le (1+y)y^2/2$, 
where the inequality is trivial for $y=1$ due to $\varphi(-1)=1$. 
By Lemma~\ref{lem:varphi} we have $\eps^2/2 \le \varphi(-\eps)$, 
so that 
\[
 \varphi(-A\eps) \le (1+A\eps)(A\eps)^2/2 \le (1+A\eps)A^2\varphi(-\eps) .
\]

Turning to the second inequality of \eqref{eq:varphi3} we henceforth 
assume $\gamma > 0$ and $\eps \in [0,1)$, as the claim is trivial 
otherwise. Let $\rho(x)=\varphi(-x)$, and note that 
$\rho'(x) = -\log(1-x)$ and $\rho''(x) = 1/(1-x)$. Since 
$\log(1-x) \ge -x/(1-x)$ for $x \in [0,1)$, c.f.\ \eqref{eq:taylor:log}, 
we see that 
$\rho'(\eps) \le \eps/(1-\eps)$. Note that $\gamma>0$ and 
$3\sqrt{\gamma} \le 1-\eps$ imply $0 < 3 \gamma^{3/2} \le \gamma - \gamma \eps \le 1-(1+\gamma)\eps$. 
So, recalling $\eps^2/2 \le \varphi(-\eps)$ and $A=1+\gamma$, using 
Taylor's theorem with remainder it follows that 
$0\le A\eps<1$ and 
\begin{equation*}\label{eq:varphi3gen}
\begin{split}
\varphi(-A\eps) & \le \varphi(-\eps) + \gamma \eps^2/(1-\eps) + (\gamma \eps)^2/[2(1-(1+\gamma)\eps)] \\
& \le \left(1+2 \gamma/(1-\eps) + \gamma^2/(1-(1+\gamma)\eps) \right) \varphi(-\eps) \le (1+\sqrt{\gamma}) \varphi(-\eps),
\end{split}
\end{equation*}
completing the proof of \eqref{eq:varphi3}. 
\end{proof}
\begin{proof}[Proof of Lemma~\ref{lem:lambda:HG}]
Define $\cS_H$ as the collection of all non-isomorphic subgraphs 
$J \subseteq H$ with $e_J \ge 1$. Let $N(n,H)$ denote the number of 
copies of $H$ in $K^{(k)}_n$. Note that $N(n,H)=\Theta(n^{v_H})$.
By double counting pairs $(J',H')$ of 
copies of $J$ and $H$ with $J' \subseteq H' \subseteq K^{(k)}_n$, 
using symmetry we infer that, in $K^{(k)}_n$, there are exactly 
\begin{equation}\label{eq:HG:lambda:JH}
\lambda_{J,H}(n) = \frac{N(n,H) C_{J,H}}{N(n,J)} = \Theta(n^{v_H-v_J}) 
\end{equation}
copies of $H$ containing any given copy of $J$. Since 
$\E X_J=N(n,J) p^{e_J}$ and $C_{H,H}=1$, by distinguishing all 
possible intersections of $H$-copies it follows that 
\begin{equation}\label{eq:HG:lambda:pairs}
\Lambda(X_H) \le \E X_H + \sum_{J \in \cS_H: J \neq H} N(n,J) \lambda^2_{J,H}(n) p^{2 e_H-e_J} = \sum_{J \in \cS_H} C^2_{J,H} \frac{(\E X_H)^2}{\E X_J} .
\end{equation}
Recall that $\E X_J = \Theta(n^{v_J}p^{e_J})$. By definition, for 
every $K \in \cS_H \setminus \cI_H$ there is $J \in \cI_H$ with 
$v_J=v_{K}$ and $e_J \ge e_K+1$. Using $\E X_{K} = \Omega(p^{-1}\E X_{J})$ 
we infer 
\begin{equation}\label{eq:HG:lambda:pairs:UB}
\Lambda(X_H) \le \sum_{J \in \cI_H} \bigl(1+\indic{e_J \ge 2}O(p)\bigr) C^2_{J,H} \frac{(\E X_H)^2}{\E X_J} .
\end{equation}
Suppose that $\omega=\omega(n) \to \infty$ satisfies 
$1 \le \omega \le n^{1/(2m_k(H)+1)}$. 
Using $m_k(H) \ge (e_K-1)/(v_K-k)$ when $e_K\ge2$, 
note that for $p \ge \omega n^{-1/m_k(H)}$ we have
\begin{equation}\label{eq:HG:lambda:mk}
\begin{split}
\min_{K \in \cS_H:v_K>k}n^{v_K-k}p^{e_K-1} 
& \ge \min\bigl\{n,\min_{K \in \cS_H: e_K \ge 2}\omega^{e_K-1} \bigr\} \ge \omega .
\end{split}
\end{equation}
Thus the `edge-term' with $e_J=1$ and $v_J=k$ dominates 
\eqref{eq:HG:lambda:pairs:UB} for $p \ge \omega n^{-1/m_k(H)}$: 
indeed, $K \neq J$ implies $\E X_{K} = \Omega(\omega \E X_{J})$. 
As $\omega n^{-1/m_k(H)} \le \omega^{-1}$, the $1+\indic{e_J \ge 2}O(p)$ 
factor in \eqref{eq:HG:lambda:pairs:UB} can thus be replaced by 
$1+O(\omega^{-1})$, establishing the upper bound of 
\eqref{eq:HG:lambda}. Furthermore, by combining $\E X_{K} = \Omega(p^{-1}\E X_{J})$ 
and $\E X_{K} = \Omega(\omega\E X_{J})$ in an analogous way, it is 
not difficult to see that \eqref{eq:HG:min} holds. For the lower 
bound of \eqref{eq:HG:lambda} we argue similar as for \eqref{eq:HG:lambda:pairs}, 
but restrict our attention to intersections in subgraphs $J \in \cI_H$ 
only. Moreover, to avoid overcounting (due to additional intersections 
outside of $J$), in the case $J \neq H$ we replace $\lambda^2_{J,H}(n)$ by 
\begin{equation*}
\lambda_{J,H}(n)\Bigl(\lambda_{J,H}(n)- O\bigl(\sum_{J' \subsetneq G \subseteq H:J' \cong J} \lambda_{G,H}(n)\bigr)\Bigr) = \bigl(1-O(n^{-1})\bigr)\lambda^2_{J,H}(n) ,
\end{equation*}
where we used \eqref{eq:HG:lambda:JH} and that every copy of 
$J\in \cI_H$ in $H$ is induced (which implies $v_G\ge v_J+1$). 
With these modifications, the lower bound of \eqref{eq:HG:lambda} 
follows. 
\end{proof}
\end{appendix}

\end{document}